\documentclass[letterpaper, DIV=17]{scrartcl}

\usepackage{amssymb, amsthm, mathtools,bbm}
\usepackage[square,sort,comma,numbers]{natbib}
\usepackage{comment}

\usepackage[none]{hyphenat}
\usepackage[english]{babel}
\usepackage{enumerate}
\usepackage{graphicx}%
\usepackage{subcaption}
\usepackage{dsfont}%
\usepackage{float}%
\usepackage{hyperref}
\usepackage{color}
\usepackage[utf8]{inputenc}
\usepackage{graphicx}
\usepackage{verbatim} 
	\usepackage{hyperref}
	\usepackage{comment}
	\usepackage{xfrac} 
	\usepackage{tensor} 
	\usepackage{listings}  

	\usepackage{pgfplots}              
	\pgfplotsset{compat=newest}
	\usepgfplotslibrary{fillbetween}

	\usepackage[title]{appendix}

	\usepackage[shortcuts]{extdash}  
	
	\newtheorem{theorem}{Theorem}[section]  
	\newtheorem{cor}[theorem]{Corollary}
	\newtheorem{prop}[theorem]{Proposition}
	\newtheorem{lemma}[theorem]{Lemma}
	\newtheorem{conjecture}[theorem]{Conjecture}

	\numberwithin{equation}{section}
	\theoremstyle{definition}

	\newcommand{\E}{\mathbb E}
	\newcommand{\Prob}{\mathbb P}
	
	\newcommand{\N}{\mathbb N}
	
	\newcommand{\Z}{\mathbb Z}
	\newcommand{\R}{\mathbb R}

	\newcommand{\dist}{\operatorname{dist}}
	\newcommand{\distZ}[1]{\left\|#1\right\|}

\begin{document}
		\title{Digit Mixing under Polynomial Maps}
		
		\author{Chokri Manai$^{1}$ \\
			\small $^1$ Courant Institute, New York University, USA \\[-.5ex]
			}
		
		\date{\small \today \\[-.5cm]}
		\maketitle		
		
		\begin{abstract}
		Let $X=\sum_{n\geq1}\xi_n2^{-n} $ be a random number where we model  the digits $\xi_n$ as independent Bernoulli random variables with possibly non-identical parameters $p_n=\mathbb{P}(\xi_n=1)$.  For any polynomial $P\in\mathbb{R}[X]$ with degree $d\geq2$, we prove almost sure absolute normality of $P(X)$ under the condition $p_n(1-p_n)\geq (\log n)^{\Gamma(d)} n^{-(d-1)/d}$ for a suitable constant $\Gamma(d)$ depending only on the degree $d$. Our analysis reveals the sharp power law $n^{-(d-1)/d}$, which is suggested by an elementary heuristics regarding carrier interactions. Our results establish a transition as we further show the pure critical power law is insufficient, but the precise critical window remains an interesting open problem. As far as we know, this is the first sharp result on digit mixing. We complement our main results by structurally convenient summability criteria, which turns out to be sharp at least for $X^2$, and we formulate a more general conjecture for higher degrees.  Our  proofs rely on Fourier decay estimates which we obtain by probabilistic argument involving conditioning and non-resonancy estimates combined with a subtle triangularization argument. \\
			\noindent
			{\small \textbf{Keywords:} Digit mixing,  normal numbers, Fourier decay estimates  \\
				\noindent
				\textbf{MSC:}  11K16; 11K06; 42A38; 60F15}
		\end{abstract}
		
		\section{Introduction}
		
		A real number $y$ is normal in an integer base $b\ge2$ if every block of $k$ base-$b$ digits occurs in the expansion of $y$ with asymptotic frequency $b^{-k}$. It is absolutely normal if it is normal in every integer base.  Borel's classic theorem states that   (Lebesgue-)almost every real number is absolutely normal, but proving normality for explicitly given numbers remains notoriously difficult \cite{Borel09}.  In particular, Borel's intriguing conjecture that every irrational algebraic number is normal remains widely open. We find it a fruitful research direction to establish a probabilistic version of Borel's conjecture. Consider a random number $X$ which is "almost" an integer in the sense that non-zero digits are very rare. Borel's conjecture suggests that for example $\sqrt{X}$ should still be almost surely normal. Heuristically, we expect rapid digit mixing under algebraic roots which ``forget"" the rigid integer input. In this work, we restrict ourselves to the simpler, but already very rich, case of polynomial maps for which we give a rigorous foundation of the digit mixing phenomenon.
		
		We consider the following probabilistic model. Let
		\begin{equation} X=\sum_{n\ge1}\xi_n2^{-n}
		\end{equation}
		where the digits \(\xi_n\) are independent Bernoulli random variables with possibly non-identical parameters \(p_n=\Prob(\xi_n=1)\). For convenience, our probabilistic model starts from a binary expansion, but our method applies to analogous construction in different bases, too. Moreover our analysis translates to the bulk digit distribution  of large random integers.  Note that if the digits are asymptotically biased $\lim_{N \to \infty} \frac1N \sum_{n = 1}^{N} p_n \neq \frac12$, then $X$ is not normal in base 2. Nevertheless nonlinear operations should create carries and mix the digits.
		For a polynomial $P$ of degree $d \geq 2$, one might expect that the leading term 
		$$ X^d = \sum_{n = 1}^{\infty} \sum_{\substack{1\leq i_1, \cdots, i_d\\ i_1 + \cdots + i_d=n}} \xi_{i_1} \cdots\xi_{i_d} 2^{-n} $$
		dominates the mixing behavior.  $X^d$ can only be normal, if the $n$-level interaction $\sum_{\substack{1\leq i_1, \cdots, i_d\\ i_1 + \cdots + i_d=n}} \xi_{i_1} \cdots\xi_{i_d}$ is typically non-zero. Since there are order $n^{d-1}$ terms, this suggests that a transition occurs at the power law $p_n \sim n^{-\frac{d-1}{d}}$. Our main results below, Theorem~\ref{thm:main} and Theorem~\ref{thm:necessary}, largely confirm this picture. 
		
		Let us introduce $v_n := p_n(1-p_n) \in [0, 1/4]$ measuring the randomness of the $n$-th digit.

		\begin{theorem}
			\label{thm:main}
			Let \(P\in\R[x]\) have degree \(d\ge2\).
			There is a constant $\Gamma_0=\Gamma_0(d)$ only depending on the degree of $P$ such that the following holds. Suppose that for all sufficiently large $n$,
			\begin{equation}\label{eq:sharp}
				v_n\ge (\log n)^{\Gamma_0} n^{-(d-1)/d}.
			\end{equation}
			Then $P(X)$ is  almost surely absolutely normal.
		\end{theorem}
		
		Very naturally, we need to exclude linear polynomials $ax +b$. The lack of a nonlinear term makes the mixing behavior coefficient dependent and this is a completely different phenomenon. The proof of Theorem~\ref{thm:main} can be found in Section~\ref{sec:proof2}. We rely on the common approach to derive a decay estimate for the Fourier transform of $P(X)$, which implies normality via Weyl's criterion. Compared to analytic methods in the literature, our proof relies on probabilistic ideas which seem to suit this problem particularly well. The main technical challenge is to control the higher dimensional interactions, which are per se highly correlated.  We address this issue by a subtle triangular decomposition. This allows us to proceed via an iterative conditioning argument which reveals the desired Fourier cancellations.  Theorem~\ref{thm:main} meets the expected exponent $-(d-1)/d$, but it is natural to wonder whether the logarithmic term in \eqref{eq:sharp} is also sharp. Our proof requires an exponent $\Gamma_0$ which grows linearly in the degree $d$. This is most likely an artifact of our method. We believe however that a logarithmic correction might be needed.  Our next result shows that  at least $v_n \sim C n^{-\frac{d-1}{d}}$ is not enough to guarantee normality of $P(X)$. Here, we used the common notion $a_n \sim b_n$ if $\lim_{n \to \infty} \frac{a_n}{b_n} = 1.$
		
		\begin{theorem}\label{thm:necessary}
			Let \(P\in\mathbb Q[x]\) be of degree \(d\ge 2\).  If
			\begin{equation}\label{eq:necessary}
				p_n\sim C n^{-(d-1)/d}
				\qquad (0<C<\infty),
			\end{equation}
			then \(P(X)\) is almost surely not normal in base $2$.  In particular,
			\(P(X)\) is almost surely not absolutely normal.
		\end{theorem}
		
		We suppose the exact asymptotics $	p_n\sim C n^{-(d-1)/d}$ to simplify proof in Section~\ref{sec:proof3}. However, one can easily show that this extends to all sequences for which $v_n \leq C n^{-(d-1)/d}$ holds for sufficiently large $n$.  Note that we only consider rational polynomials. The reason is simple. Given a real number $x$, then for Lebesgue a.e. tuples of polynomial coefficients the number $P(x)$ is absolutely normal. For any fixed sequence of probabilities $p_n$, Fubini's theorem implies that $P(X)$ is almost surely normal for Lebesgue-a.e. polynomial $P$. Hence the exceptional set of polynomials will always be small and Theorem~\ref{thm:main} can be seen as lift of the a.e. statement to an assertion on all (and in particular rational) polynomials. 
		
		It seems to be an interesting but difficult problem to better understand  the critical window for the decay of $p_n$ for which $P(X)$ undergoes a transition from a non-normal number to an absolutely normal number. As a first step, one would need to understand the optimal decay rate for $p_n$ which turns $P(X)$ to an almost surely normal number. In a second step, one might attempt to find nontrivial digit distributions in the critical window. Perhaps closer to the original motivation of this work, is to consider extensions of Theorem~\ref{thm:main} to analytic maps. Very naively, this corresponds to $d = \infty$ and Theorem~\ref{thm:main} suggests (up to corrections) a power law $n^{-1}$ as threshold for $p_n$. Note that this cannot be further improved as $p_n \sim n^{-r}$ for $r > 1$ is summable and, thus, $X$ would be almost surely rational. Our current methods are not strong enough to cope with general analytic maps. The class of analytic maps is quite rich and it is unclear to us if, say, $\sqrt{\cdot}$ and $\exp(\cdot)$ showcase the same mixing behavior. We consider this of fundamental interest as a difference could point to different levels of difficulty in attempting a normality proof of $e$ and $\sqrt{2}$.

		\subsection{A summability criterion}
		
	   We want to discuss a further, arguably more elegant, approach based on a sufficient summability criterion.  Let us consider first the special case $X^2$, which is of particular interest.  Our second sufficient criterion for digit mixing of $X^2$ is based on the simple interaction function 
		\begin{equation}\label{eq:R2def}
			R_2(s) :=\sum_{\substack{1\le i<j\\ i+j=s}}v_iv_j.
		\end{equation}
		If there are no such pairs, the sum is interpreted as $0$.  One can think of $R_2(s)$  to measure the carrier mixing around the binary position $s$. Following this vague idea, large values of $R_2(s)$ should lead to the normality of $X^2$. This is encoded in the following summability criterion.
		\begin{theorem}\label{thm:quadratic}
			There is an absolute constant \(c_2>0\) with the following property.  If
			\begin{equation}\label{eq:R2summability}
				\sum_{s\ge1}s\exp(-c_2R_2(s))<\infty,
			\end{equation}
			then \(X^2\) is almost surely absolutely normal.
		\end{theorem}
		
		The proof of Theorem~\ref{thm:quadratic} can be found in Section~\ref{sec:proof1}. It is based on similar ideas as our argument for Theorem~\ref{thm:main}, but it is much shorter and simpler. We invite the reader to first consult this proof as it displays the main ideas in a very clear way. Theorem~\ref{thm:quadratic} covers a wide range of decaying probabilities $p_n$ under which normality of $X^2$ is  guaranteed. In particular, it gives an alternative route to Theorem~\ref{thm:main} for $P(X) = X^2$
		
		\begin{cor}\label{cor:quadratic-log}
			There exists an absolute constant \(A_0>0\) such that if, for all sufficiently large \(n\),
			\begin{equation}\label{eq:quadlogcondition}
				v_n\ge A\left(\frac{\log n}{n}\right)^{1/2}
			\end{equation}
			with \(A\ge A_0\), then \(X^2\) is almost surely absolutely normal.
		\end{cor}

		The quadratic case is special in the sense that we succeeded in formulating a clean and sharp criterion. The main advantage compared to Theorem~\ref{thm:main} is that we do not need strict pointwise bounds on $v_n$ For higher degree polynomials, the situation is more involved. The most natural generalization of Theorem~\ref{thm:quadratic} to more general polynomials of degree $d \geq 2$ would consider the function
		$$R_d(s) :=\sum_{\substack{1\leq i_1 < \cdots < i_d\\ i_1 + \cdots + i_d=s}}v_{i_1} \cdots v_{i_d} $$
		and a sufficient criterion should be $\sum_{s\ge1}s\exp(-c_dR_d(s))<\infty$. We formulate this as the following conjecture.
		
		\begin{conjecture}\label{con:summability}
			 Let \(P\in\R[x]\) have degree \(d\ge2\). There is an absolute constant \(c_d>0\), depending only on the degree $d$, with the following property.  If
			\begin{equation}\label{eq:summability}
				\sum_{s \geq 1} s\exp(-c_d R_d(s))<\infty,
			\end{equation}
			then $P(X)$ is almost surely absolutely normal.
		\end{conjecture}
		
		If true, Conjecture~\ref{con:summability} would imply Theorem~\ref{thm:main} with logarithmic exponent $\Gamma_0 = \frac{d-1}{d}$. This would be particularly promising in view of analytic maps since $\Gamma_0 \to 1$ as $d \to \infty$. We are unfortunately not able to prove this. The key issue is that different tuples $ i_1 + \cdots + i_d=s$ might overlap and this causes involved dependencies. In particular, if a sparse collection of $v_n$ take large values and dominates $R_d(s)$ this problem becomes severe. Note that for $d = 2$ this issue is not present as the pairs are always disjoint. Theorem~\ref{thm:main} avoids this problem since the pointwise condition excludes ``holes" in the sequence $v_n$.
		 
		In the following, we describe a temporary way out: a weaker summability condition, which fails to be sharp, but might still be of interest.
		 For \(s\in\N\), let \(\mathfrak F_d(s)\) denote the set of all finite families \(\mathcal I\) of pairwise disjoint \(d\)-element subsets \(I\subset\N\) satisfying $\sum_{i\in I}i=s$.
		We define
		\begin{equation}\label{eq:Rddef}
			R_d'(s)=\sup_{\mathcal I\in\mathfrak F_d(s)}
			\sum_{I\in\mathcal I}\prod_{i\in I}v_i,
		\end{equation}
		with the convention \(R_d'(s)=0\) if \(\mathfrak F_d(s)\) is empty.  We stress the additional disjointness assumption compared to the more natural functional $R_d$. This gives rise to the following.  
		
		\begin{prop}\label{thm:polynomial}
			Let \(P\in\R[x]\) have degree \(d\ge2\).  Then there is a constant \(c_P>0\), depending only on \(P\), with the following property.  If
			\begin{equation}\label{eq:Rdsummability}
				\sum_{s\ge1}s \exp(-c_PR'_d(s))<\infty,
			\end{equation}
			then \(P(X)\) is almost surely absolutely normal.
		\end{prop}
		
		Proposition~\ref{thm:polynomial} offers a structurally simple sufficient criterion for non-linear polynomials. By design,  the proof of Theorem~\ref{thm:quadratic} easily extends to the polynomial setting with $R_d'$. To avoid too much repetition, we refrain from giving a full proof and only briefly sketch the needed modifications in Section~\ref{sec:proof1}. 
		The following decay rates for $v_n$  can be obtained from  Proposition~\ref{thm:polynomial}. 
		\begin{cor}\label{cor:polynomial-log}
			Let \(P\in\R[x]\) have degree \(d\ge2\).  There exists \(A_0=A_0(P)>0\) such that, if for all sufficiently large \(n\),
			\begin{equation}\label{eq:polylogcondition}
				v_n\ge A\left(\frac{\log n}{n}\right)^{\frac{1}{d}}
			\end{equation}
			with \(A\ge A_0\), then \(P(X)\) is almost surely absolutely normal.
		\end{cor}
		The proof of Corollary~\ref{cor:polynomial-log} is carried out in Section~\ref{sec:proof1}. Unfortunately,  the criterion becomes even weaker as the degree of the polynomial increases which is in stark contrast to the strengthened carrier mixing for higher degrees.

		\subsection{Overview of related literature}
		
		The study of normal numbers goes back to Borel's  theorem
		\cite{Borel09}. Weyl's classical criterion \cite{Weyl1916} is a valuable tool to prove normality via Fourier methods. The
		Davenport--Erd\H{o}s--LeVeque (DEL) criterion \cite{DELeV63} is a further standard  tool in this context as it turns second-moment estimates into almost sure equidistribution.  We shall use these ideas throughout the paper.  Classical references for uniform
		distribution and normality include \cite{KN74,DT97,Bug12}. Our results give rise to non-normal numbers $x$ for which $P(x)$ is normal. This mirrors to some extent classical results on normality in different integer bases \cite{Cass59} and Schmidt \cite{Schmidt1960,Schmidt1962}. The dimension of numbers with prescribed (non-normal) digit distribution was studied by Eggleston\cite{Eggleston1949}.
		
		A different line of research concerns the explicit construction of normal numbers. This is a surprisingly non-trivial problem. Champernowne's constant \cite{Champernowne1933} $0.1235617910111213\ldots$ consecutively concatenating all positive integers in base-$10$ is probably the most well-known normal number in base $10$. 
		The celebrated results by Copeland, Davenport, Erd\H{o}s \cite{CE46, DE52} extend this type of construction to the concatenation of primes and polynomials in a given base. The construction of absolutely normal numbers is more involved. The first construction goes back  Sierpiński \cite{Sir17} using a clever cover argument. Becher and Figueira \cite{BF02} turned this idea into an algorithm. Later, very efficient algorithms computing absolutely normal numbers have been found \cite{BHS13, LM21}.
		
		 Levin's construction \cite{Lev79} is noteworthy as it constructs numbers with particularly low discrepancy number in a given base. There is also construction which slightly beats Fukuyama's law of the iterated  logarithm \cite{Fu08} jointly in all integer bases \cite{ABSS17}. Bugeaud and Korobov raised the question about the best possible bound on the discrepancy numbers and this remains an interesting challenge. Recent work provide methods to construct numbers which make countably many functions normal \cite{ManaiTranscendence} and non-normal, respectively \cite{ManaiSuperDense}. A natural next step would be an intrinsic algorithm constructing a number $x$ for which $P(x)$ is normal and $Q(x)$ is non-normal - provided that the polynomials $P$ and $Q$ are linearly independent.
		
		A motivating background is Borel's conjecture that every irrational
		algebraic number is normal to every integer base.  At the moment, only very partial results in this direction are known, see for instance the discussion in
		\cite{BBCP04,AB11, BC02}.  Besicovitch's theorem on the decimal digits of
		squares \cite{Besicovitch35} is an early example where polynomial structure and digit distribution interact in a non-trivial way.  The present work may be viewed as a probabilistic model for such digit-mixing phenomena. Conversely, one might also ask which operations preserve normality. Rauzy's result on deterministic numbers \cite{Rauzy76} gives a complete solution for the additive problem $x + a$, but for the multiplicative version $ax$ it remains open whether only rational $a$ preserve normality \cite{BD25}.
		
		\medskip
		Our work has also close ties with several seminal works in fractal and harmonic analysis. We can only mention a selection of these rich fields. The law of $X=\sum_{n\ge1}\xi_n2^{-n}$ can also be understood as infinite convolution of ``coin-tossing'' measures. 
		Such measures are naturally related to the classical theory of infinite product measures, going back to Kakutani \cite{Kakutani1948};
		and were already studied in the context of singular Rajchman measures
		by Hartman--Kershner \cite{HK38}. Natural questions about these types of measures concern the Fourier decay and normality of $X$ \cite{Bis03,Bis04,GMSZ20}. Noteworthy, are also Lyon's result \cite{Lyons1986} showing that there are Rajchman measure $\mu$ for which $X$ is not $\mu$-a.s normal and Host's theorem on $\mu$-a.s. normality for some ergodic measures \cite{Host1995}.
		The Fourier-analytic side of the paper is also reminiscent of the theory of Bernoulli convolutions. Probabilistically speaking,  one studies a random number $X=\sum_{n\ge1}\xi_n \lambda^{-n}$, where the Bernoulli random variables $\xi_n$ are unbiased and the behavior in terms of the parameter $\lambda \in \R$ is of key interest. Classical and modern references include Erd\H{o}s' work \cite{Erdos39}, the theorem of Solomyak \cite{Solomyak1995}, the survey of Peres--Schlag--Solomyak \cite{PSS00}, and later developments
		such as \cite{PS96,Shmerkin2019}.  Kahane's monograph \cite{Kahane1985} provides a broad background on random series.

	  	However, the both mentioned works are primarily concerned with the measure $\mu$ itself. Closest to our setting, are beautiful recent works study nonlinear pushforward \(P_\#\mu\) as we do. Here, the typical setting is that $\mu$ is a self-similar or fractal measure.
		Hochman--Shmerkin \cite{HS15} proved equidistribution results from fractal measures, and
		subsequent works such as \cite{ABS22,AHW11} established pointwise normality
		results under Fourier decay assumptions for self-similar or self-conformal measures.  Baker--Banaji \cite{BB25} studied polynomial Fourier decay for fractal measures and their pushforwards for even analytic maps. Despite being conceptually close, these results have a quite distinct flavor. Indeed, coin-toss measures $\mu$ with constant probability $p_n = p$ are the only case for which the results in \cite{BB25} and in this work apply. Clearly this is in both settings a simple limiting case. The difference is also eminent in the proofs which in our case are more probabilistic compared to rather analytic strategies in the aforementioned works. Perhaps, a combination of our ideas and fractal analytic methods may lead to further progress.

		\section{Proof of the summability criterion}\label{sec:proof1}
		Throughout our proofs, we abbreviate Fourier modes 
		\[ e(u) :=e^{2\pi i u},
		\]
		for $u \in \R$. Moreover, it is convenient to work with the lattice norm 
		$$ \|u\|:=\dist(u,\Z) $$
		measuring the distance to the integers $\Z$. The natural numbers with and without $0$ are denoted by $\N_0$ and $\N$, respectively. For a set $A$, we write $|A|$ for its cardinality.  We follow the convention that universal constants $C, c > 0$ may change from line to line. We also use Landau's $O$-notation, i.e. for two positive sequences $b_n, a_n$, we write $b_n = O(a_n)$ if there exists a constant $C$ such that $b_n \leq C a_n$ for all $n$. If $b_n = o(a_n)$, we further have $\limsup \frac{b_n}{a_n} = 0$.

		\subsection{Preliminaries}
		
		In this section, we collect a few elementary estimates which we use in the sequel. We start with a simple lower bound for the $\sin$-function.
		
		\begin{lemma}\label{lem:sinbound}
			For all \(u\in\R\),
			\[
			\sin^2(\pi u)\ge 4\|u\|^2.
			\]
			Consequently, for \(0\le a\le1\),
			\[
			1-a\sin^2(\pi u)\le \exp(-4a\|u\|^2).
			\]
		\end{lemma}
		
		\begin{proof}
			By periodicity and symmetry it is enough to consider \(0\le u\le1/2\).  On this interval \(\sin(\pi u)\) is concave, equals \(0\) at \(0\), and equals \(1\) at \(1/2\).  Hence it lies above the chord joining these points, so \(\sin(\pi u)\ge2u=2\|u\|\).  Squaring gives the first claim.  The second follows from \(1-x\le e^{-x}\).
		\end{proof}
		
		Our proof frequently compares different dyadic scales. The following trivial bound turns out to be useful.
		
		\begin{lemma}\label{lem:scale}
			Let \(\lambda\ne0\) and assume \(|\lambda|\ge1\).  Put $
			s(\lambda)=\lfloor\log_2|\lambda|\rfloor+2$. 
			Then $ \frac14\le |\lambda|2^{-s(\lambda)}<\frac12$,
			and therefore
			$$ \|\lambda2^{-s(\lambda)}\|\ge\frac14 .$$
		\end{lemma}
		
		A slightly more sophisticated scale estimate will be content of the next result. For \(\alpha>0\), define the scale map
		\begin{equation}\label{eq:Salpha}
			S_\alpha(u)=
			\begin{cases}
				0, & 0\le \alpha u<1,\\
				\lfloor\log_2(\alpha u)\rfloor+2, & \alpha u\ge1.
			\end{cases}
		\end{equation}
		The value \(0\) for small frequencies is motivated by the later employed second-moment argument.
		
		\begin{lemma}\label{lem:counting}
			Fix \(\alpha>0\), an integer \(b\ge2\), and \(h\in\Z\setminus\{0\}\).  There is a constant \(C=C(\alpha,b,h)\) such that for every \(s\ge1\),
			\[
			|\{(m,r)\in\N_0\times\N:S_\alpha(|h|b^m(b^r-1))=s\}|\le C(s+1).
			\]
			Moreover the set of pairs \((m,r)\) for which \(S_\alpha(|h|b^m(b^r-1))=0\) is finite.
		\end{lemma}
		
		\begin{proof}
			First suppose \(S_\alpha(|h|b^m(b^r-1))=s\ge1\).  Then
			\[
			2^{s-2}\le \alpha |h|b^m(b^r-1)<2^{s-1}.
			\]
			In particular \(b^m\le C_ 2^s\), so there are at most \(C(s+1)\) possible values of \(m\).  For each fixed \(m\), the same inequality places \(b^r-1\) in an interval of the form
			\[
			c_1 2^s b^{-m}\le b^r-1< c_2 2^s b^{-m},
			\]
			where \(c_1,c_2>0\) depend only on \(\alpha,b,h\).  This immediately implies that there are at most $C>0$ different values of $r$ satisfying this bound.  In total, this yields the claimed \(O(s)\) bound.
			
			If \(S_\alpha(|h|b^m(b^r-1))=0\), then \(\alpha |h|b^m(b^r-1)<1\).  Since \(b^m(b^r-1)\to\infty\) whenever either \(m\to\infty\) or \(r\to\infty\), only finitely many pairs satisfy this inequality.
		\end{proof}
		
		\subsection{A Fourier criterion for absolute normality}
		
		It is well known that a fast enough Fourier decay of a random number implies its almost sure absolute normality. Most results in the literature start from polynomial decay and the key point of the following proposition is that a much slower decay is enough.
		
		\begin{prop}\label{prop:normalitycriterion}
			Let \(Y\) be a real-valued random variable.  Let \(\alpha>0\), let \((B_s)_{s\ge1}\) be a nonnegative sequence, and let \(c>0\) be some constant.  Assume that for every real \(u\) with \(S_\alpha(|u|)\ge1\),
			\begin{equation}\label{eq:FourierAssumption}
				|\E e(uY)|\le \exp(-cB_{S_\alpha(|u|)}).
			\end{equation}
			Assume also that
			\begin{equation}\label{eq:Bsum}
				\sum_{s\ge1}s\exp(-cB_s)<\infty.
			\end{equation}
			Then \(Y\) is almost surely absolutely normal.
		\end{prop}
		
		The proof is more or less a routine application of Weyl's criterion combined with a second moment analysis. To be self-contained, we still choose to present the straightforward proof.
		\begin{proof} 
			Fix an integer base \(b\ge2\) and an integer \(h\ne0\).  Define
			\[A_N(h,b) :=\frac1N\sum_{n=0}^{N-1}e(hb^nY).\]
			We want to prove that \(A_N(h,b)\to0\) almost surely.  Expanding the second moment gives
			\begin{align*}
				\E |A_N(h,b)|^2
				&=\frac1{N^2}\sum_{m,n=0}^{N-1}
				\E e(h(b^n-b^m)Y).
			\end{align*}
			The diagonal terms contribute \(1/N\).  For \(n>m\), put $r:=n-m\ge1$.  Then $b^n-b^m=b^m(b^r-1)$ and let us further introduce the abbreviation
			$u_{m,r}=|h|b^m(b^r-1)$.
			The pairs with \(S_\alpha(u_{m,r})=0\) are finite by Lemma \ref{lem:counting}; their total contribution to the infinite off-diagonal sum is finite because each term has modulus at most \(1\).  For the remaining pairs, \eqref{eq:FourierAssumption} gives
			\[
			\left|\E e(hb^m(b^r-1)Y)\right|
			\le \exp(-cB_{S_\alpha(u_{m,r})}).
			\]
			Using Lemma \ref{lem:counting},
			\begin{align*}
				\sum_{m\ge0}\sum_{r\ge1}
				\left|\E e(hb^m(b^r-1)Y)\right|
				&\le C_{\alpha,b,h}+\sum_{s\ge1}C_{\alpha,b,h}(s+1)e^{-cB_s}<\infty.
			\end{align*}
			It follows that there is a constant \(C_{b,h}<\infty\) such that, for every \(N\),
			\begin{equation}\label{eq:ANsecond}
				\E |A_N(h,b)|^2\le \frac1N+\frac{C_{b,h}}{N^2}\le \frac{C_{b,h} +1}N.
			\end{equation}
			We argue via a well-known thinning argument. Choose \(N_k=k^2\).  Then
			\[
			\sum_{k=1}^\infty \E |A_{N_k}(h,b)|^2<\infty.
			\]
			Hence, by Markov's inequality and Borel--Cantelli,
			$A_{N_k}(h,b)\to0$ almost surely.
			If \(N_k\le N<N_{k+1}\), then
			\begin{align*}
				|A_N(h,b)|
				&\le |A_{N_k}(h,b)|+\frac{N-N_k}{N}+\frac{|N-N_k|}{N}|A_{N_k}(h,b)|
				\le |A_{N_k}(h,b)|+2\frac{N_{k+1}-N_k}{N_k}.
			\end{align*}
			Since \((N_{k+1}-N_k)/N_k=(2k+1)/k^2\to0\), we get \(A_N(h,b)\to0\) almost surely.
			
			Taking the countable intersection over all \(b\ge2\) and all \(h\in\Z\setminus\{0\}\), we obtain, almost surely,
			\[
			\frac1N\sum_{n=0}^{N-1}e(hb^nY)\to0
			\]
			for every such pair \((b,h)\).  Weyl's criterion then implies that \(\{b^nY\}_{n\ge0}\) is uniformly distributed modulo \(1\) for every \(b\ge2\).  This is equivalent to normality of \(Y\) in base \(b\), and hence to absolute normality.
		\end{proof}
		
		In the next sections, we will show that $X^2$ satisfies the assumptions from Proposition~\ref{prop:normalitycriterion} and is thus almost surely normal.

		\subsection{The quadratic Fourier estimate and proof of Theorem~\ref{thm:quadratic}}
		
		The key technical contribution to the proof of Theorem~\ref{thm:quadratic} is the following Fourier estimate.
		
		\begin{prop}\label{prop:quadraticFourier}
			Let \(R_2(s)\) be defined by \eqref{eq:R2def}.  There is an absolute constant \(c_2>0\) such that, for every real \(t\) satisfying \(|2t|\ge1\),
			\begin{equation}\label{eq:quadraticFourier}
				|\E e(tX^2)|\le \exp\bigl(-c_2R_2(S_2(|t|))\bigr),
			\end{equation}
			where \(S_2(|t|)=S_2^{\ast}(|t|):=\lfloor\log_2(2|t|)\rfloor+2\).  For \(|2t|<1\) the trivial estimate \(|\E e(tX^2)|\le1\) holds.
		\end{prop}
		
		Let us first record that this readily finishes the proof of Theorem~\ref{thm:quadratic}. \begin{proof}[Proof of Theorem \ref{thm:quadratic}]
			Apply Proposition \ref{prop:quadraticFourier} with $Y=X^2$, \(\alpha=2\), and \(B_s=R_2(s)\).  For frequencies with \(S_2(|u|)=0\) the Fourier transform is bounded trivially by \(1\), and these frequencies contribute only finitely many off-diagonal terms in Proposition \ref{prop:normalitycriterion}.  The summability hypothesis \eqref{eq:R2summability} is exactly \eqref{eq:Bsum}.  Proposition~\ref{prop:normalitycriterion} therefore gives almost sure absolute normality of $X^2$.
		\end{proof}
		
		The main idea behind the proof of Proposition~\ref{prop:quadraticFourier} is to identify enough sources of randomness near the digit position $s=\lfloor\log_2(2|t|)\rfloor+2$ to cause the needed cancellations.  To do so, we consider a thinned set of pairs $i + j = s$ which yields a nested structure upon we can iteratively condition and then find in each step a further linear source of randomness.
		
		\begin{proof}[Proof of Proposition~\ref{prop:quadraticFourier}]
			Assume \(|2t|\ge1\), and put as before $s=\lfloor\log_2(2|t|)\rfloor+2$. By Lemma \ref{lem:scale},
			\begin{equation}\label{eq:quadnonres}
				\|2t2^{-s}\|\ge\frac14.
			\end{equation}
			We define the set of pairs with sum $s$,
			\[ \mathcal P_s :=\{\{i,j\}:1\le i<j,\, i+j=s\}.\]
			The pairs in \(\mathcal P_s\) are pairwise disjoint, since an index $i$ can only be paired with $s-i$.
			To decouple the tail digits from those around position $s$, we fix an integer $K$ so large that
			\begin{equation}\label{eq:Kdef}
				\sum_{\ell\ge1}2^{-K\ell}\leq \frac{1}{64}.
			\end{equation}
			We partition the set of pairs $P_s$ according to the residue class of the smaller index modulo $K$. By the pigeon hole principle, we find a residue class $\mathcal P=\{(i_r,j_r)\}_{r=1}^L$ with
			\begin{equation}\label{eq:Residuedef}
			\sum_{r=1}^L v_{i_r}v_{j_r}\geq \frac{R_2(s)}{K}.
			\end{equation}
			We condition on all binary digits outside the pairs from $\mathcal P$ and treat those as constant in the following. Let us write $x_r=\xi_{i_r}$, $y_r=\xi_{j_r}$. Then, by the binomial identity
			\begin{equation}\label{eq:binom}
			tX^2=\phi(x)+\psi(y)+\sum_{r,k=1}^L D2^{i_k-i_r}x_ry_k,
			\end{equation}
			with some functions $\phi, \psi$ and the abbreviation $D :=2t2^{-s}.$ By Jensen's inequality in the $y$-variables,  we obtain 
			$$ |\mathbb E_{x,y}e(tX^2)|^2 \leq \mathbb E_{y} \left| \mathbb E_{x} e\left(\phi(x) + \sum_{r,k=1}^L D2^{i_k-i_r}x_ry_k \right) \right|^2. $$
			In a next step, we expand the square by introducing an independent copy $x'$ of $x$
			$$ \left|\mathbb E_{x} e\left(\phi(x) + \sum_{r,k=1}^L D2^{i_k-i_r}x_ry_k \right)\right|^2 = \mathbb E_{x,x'}  e\left(\phi(x) - \phi(x') + \sum_{r,k=1}^L D2^{i_k-i_r}(x_r - x'_r) y_k \right).$$
			Taking the absolute value, we arrive at the bound 
			\begin{equation}\label{eq:factor}
			|\mathbb E_{x,y}e(tX^2)|^2
			\leq \mathbb E_{\Delta} \prod_{k=1}^L
			\left|\mathbb E_{y_k}e(H_k(\Delta)y_k)\right|,
			\end{equation}
			where we introduced \(\Delta_r :=x_r-x'_r\) and
			$H_k(\Delta) :=\sum_{r=1}^L D2^{i_k-i_r}\Delta_r.$ Note that the upper bound in \eqref{eq:factor} factorizes and it remains to exclude possible resonances.
			Lemma~\ref{lem:sinbound} implies the  following elementary  Bernoulli bound,
			$$ |\mathbb E_{y_k}e(uy_k)|\leq \exp(-c v_{j_k}\|u\|^2),$$
			which in turn yields
			\begin{equation}\label{eq:quadbound}
			|\mathbb E_{x,y}e(tX^2)|^2
			\leq \mathbb E_{\Delta} \exp\left(-c\sum_{k=1}^L v_{j_k}\|H_k(\Delta)\|^2\right).
			\end{equation}
			Let us decompose $H_k$ as follows
			\[ H_k :=R_k+D\Delta_k+T_k,\]
			where \(R_k\) is measurable with respect to
			\(\Delta_1,\ldots,\Delta_{k-1}\), and
			$|T_k| \leq |D|\sum_{\ell\ge1}2^{-K\ell} \leq \frac{1}{64}$ by \eqref{eq:Kdef} and Lemma~\ref{lem:scale}.
			Recall further that by \eqref{eq:quadnonres} \(\|D\|\ge1/4\). Therefore not all three points $R_k, R_k + D, R_k - D$ can be within distance $< \frac{1}{16}$ from $\Z$. 
			The law of $\Delta_k$ the implies the conditional bound
			\[ \mathbb P\left(\|R_k+D\Delta_k\|\ge1/16
			\mid \Delta_1,\ldots,\Delta_{k-1}\right)
			\ge v_{i_k}.
			\]
			Let $I_k := \mathbbm{1}_{\|R_k+D\Delta_k\|\ge1/16}$ and note that $I_k=1$ implies $\|H_k\|\ge1/32$. Hence by \eqref{eq:quadbound}
			\begin{equation}
			|\mathbb E_{x,y}e(tX^2)|^2
			\leq \mathbb E\exp\left(-c\sum_{k=1}^L v_{j_k}I_k\right).
			\end{equation}
			Now we iterate the conditional expectations such that we obtain
			\[ \mathbb E\exp\left(-c\sum_{k=1}^L v_{j_k}I_k\right)
			\leq \prod_{k=1}^L \left(1-v_{i_k}(1-e^{-cv_{j_k}})\right)
			\leq \exp\left(-c\sum_{k=1}^L v_{i_k}v_{j_k}\right),
			\]
			and recall that the we follow the convention that $c$ may change from bound to bound. In total, we have
			\begin{equation}
			|\mathbb E_{x,y}e(tX^2)|
			\leq \exp\left(-c\sum_{k=1}^L v_{i_k}v_{j_k}\right)
			\leq \exp(-c_2R_2(s))
			\end{equation}
			for some absolute constant $c_2 > 0$. The last bound follows from the choice $\mathcal P$, that is \eqref{eq:Residuedef}.
		\end{proof}
		
		\subsection{The weak criterion for general polynomial maps}
		
		We briefly indicate how the preceding quadratic argument extends to a
		polynomial map
		\[ P(x)=a_dx^d+a_{d-1}x^{d-1}+\cdots+a_0,\qquad a_d\neq0,
		\qquad d\ge2.
		\]
		This gives the weak summability criterion stated in Proposition~\ref{thm:polynomial}.

		Put $\Lambda_P:=d!a_d$ and recall the modified $d$-fold interaction
		$$ R_d'(s)=\sup_{\mathcal I\in\mathcal F_d(s)}
		\sum_{I\in\mathcal I}\prod_{i\in I}v_i, $$
		with the set \(\mathcal F_d(s)\)  of all
		finite families \(\mathcal I\) of pairwise disjoint \(d\)-element subsets
		\(I\subset \mathbb N\) such that $\sum_{i\in I} i=s.$

		We start again by fixing the dyadic scale
		$s=\bigl\lfloor \log_2(|t\Lambda_P|)\bigr\rfloor+2$
		for $ |t\Lambda_P|\ge1$.
		Then the distance to integers is controlled by 
		\[\|t\Lambda_P2^{-s}\|\ge \frac14.\]
		In a next step,  one uses the separated-block device as in the quadratic case.
		After partitioning any family
		\(\mathcal I\in\mathcal F_d(s)\) into finitely many residue classes, depending only on $d$, one obtains a separated subfamily
		\(\mathcal I'\subset\mathcal I\) such that
		\[ \sum_{I\in\mathcal I'}\prod_{i\in I}v_i
		\geq \kappa_d \sum_{I\in\mathcal I}\prod_{i\in I}v_i
		\]
		for some $\kappa_d>0$.  Now one derives via iterated Jensen the analogue of \eqref{eq:factor} factoring the exponent into linear terms. Then one argues conditionally to arrive at a bound of the form 
		\[\left|\mathbb E e\!\left(tP(X)\right)\right|
		\leq\exp\left(-c_{P,d} \sum_{I\in\mathcal I'}\prod_{i\in I}v_i
		\right).
		\]
		The main ingredients are the the disjointness of $d$-elements in $\mathcal I$ and the carefully chosen separation which guarantees that
		\(t\Lambda_P2^{-s}\) remains uniformly away from the integers, whereas all ``future'' coefficients produced by off-diagonal interactions are small enough to be absorbed 
		 into the constants.  This is the
		$d$-linear analogue of the separated-pair decoupling used for \(X^2\).
		
		In total, one obtains with some absolute constant $c_P$
		\[\left|\mathbb E e\!\left(tP(X)\right)\right|
		\leq\exp\left(-c_P \sum_{I\in\mathcal I}\prod_{i\in I}v_i\right)
		\]
		for every \(\mathcal I\in\mathcal F_d(s)\).  Taking the supremum over
		\(\mathcal I\) yields the Fourier estimate
		\[\left| \mathbb E e\!\left(tP(X)\right) \right|
		\leq \exp\left(-c_P R_d'(s)\right),
		\qquad s=\bigl\lfloor \log_2(|t\Lambda_P|)\bigr\rfloor+2.
		\]
		 Hence, Proposition~\ref{prop:normalitycriterion} and the summability condition
		\[\sum_{s\ge1} s\exp(-c_PR_d'(s))<\infty\]
		imply almost sure absolute normality of $P(X)$.
		
		\subsection{Proof of Corollary~\ref{cor:polynomial-log}}
		
		It remains to prove Corollaries \ref{cor:quadratic-log} and \ref{cor:polynomial-log}. In fact, the quadratic case is a special case of Corollary~ \ref{cor:polynomial-log}, so we will only demonstrate this more general result. The only combinatorial ingredient is the existence of many disjoint diagonal \(d\)-tuples with all indices comparable to the diagonal parameter.
		
		\begin{lemma}\label{lem:manytuples}
			Fix \(d\ge2\).  There is a constant \(c_d >0\) such that, for every sufficiently large \(s\), there exist at least \(c_ds\) pairwise disjoint \(d\)-element sets \(I\subset\N\) satisfying
			\[ \sum_{i\in I}i=s \qquad \text{and} \qquad c_ds\leq i\leq s
			\qquad\text{for every }i\in I. \]
		\end{lemma}
		
		\begin{proof}
			Let $ M :=\left\lfloor\frac{s}{10d^2}\right\rfloor$.
			For \(\ell=1,\ldots,M\), we define
			$$ i_k(\ell) := kM+\ell,
			\qquad 1\le k\le d-1,$$
			and
			$$ i_d(\ell)=s-\sum_{k=1}^{d-1}i_k(\ell)
			=s-\frac{d(d-1)}2M-(d-1)\ell. $$
			Then, \(\sum_{k=1}^d i_k(\ell)=s\) and for \(s\) large, all these integers are positive.  More precisely, the first \(d-1\) indices are between \(M+1\) and \(dM+M\), hence are bounded from below by $\frac{1}{20d^2} s$ for large $s$.  Also,
			\[ i_d(\ell)\ge s-\frac{d(d-1)}2M-(d-1)M
			\geq s-d^2M\ge \frac{9s}{10}
			\]
			for our choice of \(M\), after adjusting constants for the floor.  Thus all indices are bounded  below by $\frac{1}{20d^2} s$
			
			The indices \(i_k(\ell)\) with \(1\le k\le d-1\) lie in the disjoint intervals \([kM+1,kM+M]\), so they are all distinct as \((k,\ell)\) varies.  The values \(i_d(\ell)\) are strictly decreasing in \(\ell\), hence are distinct.  Finally, for large \(s\), all the last indices \(i_d(\ell)\) exceed \(s/2\), whereas all the preliminary indices are less than \(s/2\).  Thus the sets
			\[
			I_\ell=\{i_1(\ell),\ldots,i_d(\ell)\},
			\qquad 1\le\ell\le M,
			\]
			are pairwise disjoint. 
		\end{proof}
		
		\begin{proof}[Proof of Corollary~\ref{cor:polynomial-log}]
			Assume \eqref{eq:polylogcondition}.  By Lemma \ref{lem:manytuples}, for every sufficiently large \(s\) there is a disjoint family \(\mathcal I_s\in\mathfrak F_d(s)\) with \(\#\mathcal I_s\ge c_ds\), and every index appearing in this family is comparable to \(s\).  Therefore, for every \(I\in\mathcal I_s\),
			\[ \prod_{i\in I}v_i \geq C_d A^d\frac{\log s}{s}\]
			for a constant \(C_d>0\).  Hence
			\[R_d'(s) \geq \sum_{I\in\mathcal I_s}\prod_{i\in I}v_i
			\geq C'_dA^d\log s.\]
			Let \(c_P\) be the contraction constant from Proposition \ref{thm:polynomial}.  If $A$ is large enough that $c_PC'_dA^d>3$, then for all sufficiently large \(s\),
			\[ s\exp(-c_PR_d'(s))\leq s\cdot s^{-3}=s^{-2}.\]
			Thus \eqref{eq:Rdsummability} holds, and Proposition \ref{thm:polynomial} applies.
		\end{proof}

		\section{Proof of Theorem~\ref{thm:main}}\label{sec:proof2}

		The proof has four parts. First we record several analytic lemmas which set the foundation for the whole argument. In particular, we sharpen our mixed difference estimate from the previous proof and provide exponential bounds for triangular multilinear forms. In the second part, we give the recipe how to extract approximate triangular forms from $d$-tuples with a sum constraint. Our analytic estimates are formulated for fair Bernoulli and in the third part, we give a simple decomposition which allows us to use bound for fair Bernoullis even in our highly biased setting. All these preparations terminate in the core Fourier decay estimate, which is the most challenging part of the argument. The proof is then again concluded by an application of  Weyl's criterion.
			
		\subsection{An analytic toolbox}
			
		We first record a variant of the well-known  Davenport--Erdos--LeVeque criterion for the convergence of series. This allows us later to establish Weyl's criterion under weaker assumptions compared to the summability criterion in Proposition~\ref{prop:normalitycriterion}.
			
		\begin{lemma}[Davenport--Erdos--LeVeque]
			\label{lem:DEL}
			Let \((Y_n)_{n\ge0}\) be complex random variables with \(|Y_n|\le1\), and consider the partial sums $S_N :=\sum_{n=0}^{N-1}Y_n$.
			If
			\begin{equation}\label{eq:DELcrit}
			\sum_{N=2}^{\infty}\frac{1}{N^3}\mathbb E|S_N|^2<\infty,
			\end{equation}
			then $ \frac{S_N}{N}\longrightarrow 0$ almost surely.
			\end{lemma}
			
		\begin{proof}
			We abbreviate  \(A_N=S_N/N\) and the assumption \eqref{eq:DELcrit} reads $\sum_{N=2}^{\infty}\frac1N\mathbb E|A_N|^2<\infty$.
			Fubini's theorem implies
			\begin{equation}\label{eq:DELcrit2}
			\sum_{N=2}^{\infty}\frac{|A_N|^2}{N}<\infty \qquad\text{almost surely.}
			\end{equation}
			We show that \eqref{eq:DELcrit2}, together with \(|Y_n|\le1\), implies \(A_N\to0\) almost surely.
				
			Fix $0<\varepsilon\leq1$ and suppose that $|A_N|\geq\varepsilon$. Let $M$ be an integer with 
			$N\le M\le \left(1+\frac{\varepsilon}{8}\right)N$. 
			Then, 
			\begin{align*}
				|A_M|&=\left|\frac{S_N+(S_M-S_N)}{M}\right|  
				\geq \frac{N}{M}|A_N|-\frac{M-N}{M}  \\
				&\geq\frac{\varepsilon}{1+\varepsilon/8}-\frac{\varepsilon/8}{1+\varepsilon/8}
				\geq \frac{\varepsilon}{2}.
			\end{align*}
			Thus every index $N$ with $|A_N|\geq\varepsilon$   forces a whole multiplicative interval of indices on which at least $|A_M|\ge\varepsilon/2$.
			If there are infinitely many indices with $|A_N|\geq\varepsilon$, we may choose an infinite disjoint subcollection of intervals $[N, \lfloor(1+\varepsilon/8)N\rfloor]$. On each chosen interval,
			$$\sum_{M=N}^{\lfloor(1+\varepsilon/8)N\rfloor}\frac{|A_M|^2}{M}\geq \frac{\varepsilon^3}{32(1+\varepsilon/8)}>0.$$
			In view of \eqref{eq:DELcrit2}, we conclude that with full probability there are only finitely many $N$ with $|A_N|\geq\varepsilon$.  Since $\varepsilon>0$ is arbitrary, $A_N\to0$ almost surely.
			\end{proof}
			
		In the  quadratic case, we used in the proof of Proposition~\ref{prop:quadraticFourier}, a kind of differencing argument to transform the coupled quadratic exponent into a  multilinear factorized expression. The following result can be seen as the theoretical framework for this approach. Let $q\geq 1$ be an integer and for \(1\le j\le q\) and \(1\le k\le L\), let \(Z_{j,k}\) be independent fair Bernoulli variables. Recall that every map $\Phi((Z_{j,k})_{j,k})$ acting on Bernoullis  has a unique multilinear representative since $Z_{j,k} \in \{0,1\}$.
			
		\begin{lemma}\label{lem:GCS}
			Let $Q$ be a real multi-variate polynomial of degree at most $q$ in the variables \(Z_{j,k}\) -- written in multilinear form. For \(1\le k_1,\ldots,k_q\le L\), let \(A(k_1,\ldots,k_q)\) be the coefficient of the monomial $Z_{1,k_1}Z_{2,k_2}\cdots Z_{q,k_q}$.
			Let \(U_{j,k}\) be independent random variables with law
			\[\mathbb P(U_{j,k}=1)=\mathbb P(U_{j,k}=-1)=\frac14, \qquad
			\mathbb P(U_{j,k}=0)=\frac12.\]
			Then
			\begin{equation}\label{eq:mixedpro}
				\left|\mathbb E e(Q(Z))\right|^{2^q}
				\leq \left|\mathbb E e\!\left(\sum_{k_1,\ldots,k_q=1}^{L}
				A(k_1,\ldots,k_q) \prod_{j=1}^{q}U_{j,k_j}\right)\right|.
				\end{equation}
		\end{lemma}
			
		\begin{proof}
			 We abbreviate $Z_j :=(Z_{j,1},\ldots,Z_{j,L})$, and set $F(Z_1,\ldots,Z_q) :=e(Q(Z_1,\ldots,Z_q))$.
			Let \(Z_j^0,Z_j^1\) be two independent copies of \(Z_j\). We claim that the following inequality
			\begin{equation}\label{eq:boxineq2}
				|\mathbb E F|^{2^q} \leq  \mathbb E_{Z^0,Z^1}
				\prod_{\omega\in\{0,1\}^q} \mathcal C^{|\omega|}
				F(Z_1^{\omega_1},\ldots,Z_q^{\omega_q}),
			\end{equation}
			holds. Here, \(\mathcal C\) is understood as the anti-linear operator $z \mapsto \bar{z}$ mapping a number to its complex conjugate. Note that the superscript $\omega_k$ identifies which of the two copies $Z^0$ or $Z^1$ is used. To prove \eqref{eq:boxineq2}, we proceed via an iterative argument. For
			$0\le r\le q$, define
			\[ \mathcal D_rF:=
			\prod_{\omega\in\{0,1\}^r}\mathcal C^{|\omega|}
			F\bigl(Z_1^{\omega_1},\ldots,Z_r^{\omega_r},Z_{r+1},\ldots,Z_q
			\bigr),
			\]
			where for $r=0$ this is interpreted as \(\mathcal D_0F=F(Z_1,\ldots,Z_q)\). 
			We claim that, for \(1\le r\le q\),
			\begin{equation}\label{eq:indbox}
			|\mathbb E \mathcal D_{r-1} F |^2 \leq \mathbb E \mathcal D_{r} F|.        
			\end{equation}
			Indeed, let \(\mathcal R_r\) denote all variables appearing in
			\(\mathcal D_{r-1}F\) except for the single vector \(Z_r\).  Then
			\[ \mathbb E \mathcal D_{r-1} F =
			\mathbb E_{\mathcal R_r} \mathbb E_{Z_r}\mathcal D_{r-1}F .\]
			By Cauchy--Schwarz,
			\[ |\mathbb E \mathcal D_{r-1} F|^2
			\leq \mathbb E_{\mathcal R_r}
			\left| \mathbb E_{Z_r}\mathcal D_{r-1}F\right|^2 .\]
			Expanding the square with two independent copies \(Z_r^0,Z_r^1\) of \(Z_r\)
			gives
			\[\mathbb E_{\mathcal R_r}\left|
				\mathbb E_{Z_r}\mathcal D_{r-1}F\right|^2
				=\mathbb E_{\mathcal R_r,Z_r^0,Z_r^1} \mathcal D_{r-1}F(\ldots,Z_r^0,\ldots)\overline{
				\mathcal D_{r-1}F(\ldots,Z_r^1,\ldots)}  
				=\mathbb E\,\mathcal D_rF
			\]
			This proves \eqref{eq:indbox}.  Iterating this bound yields
			$$ |EF|^{2^q}=|\mathcal D_0F|^{2^q}\le \mathcal D_qF.$$
			This is exactly the desired bound \eqref{eq:boxineq2}.
			
			We continue with the main proof. The right-hand side of \eqref{eq:boxineq2} equals
			$$ \mathbb E_{Z^0,Z^1}e\!\left( \sum_{\omega\in\{0,1\}^q}(-1)^{|\omega|}Q(Z_1^{\omega_1},\ldots,Z_q^{\omega_q})\right).$$
			In the alternating sum, every monomial missing at least one of the $q$ groups vanishes. Since $Q$ has degree at most $q$, only monomials containing exactly one variable from each group survive. Such a monomial $A(k_1,\ldots,k_q)Z_{1,k_1}\cdots Z_{q,k_q}$ contributes in the alternating sum
			$A(k_1,\ldots,k_q) \prod_{j=1}^{q}\bigl(Z_{j,k_j}^{0}-Z_{j,k_j}^{1}\bigr)$.
			 We observe that
			the random variables \(Z_{j,k}^{0}-Z_{j,k}^{1}\) are mutually independent and have the same distribution as $U_{j,k}$ from the statement of the lemma, which completes the proof.
			\end{proof}
			
		In the quadratic case, we were able to estimate the right-hand side of \eqref{eq:mixedpro} by a simple separation trick followed by iterative conditioning. For higher degree polynomials, the geometry of $d$-fold interactions is more involved. Moreover, the $d$-elements might overlap. The first step is to consider an ideal scenario, where these obstructions can be handled. Namely, we estimate the expectation in \eqref{eq:mixedpro} under a triangular non-degeneracy condition.
		We use the lexicographic order on the set of finite words $\{1,\ldots,L\}^{q-1}$. That is
		$(k_1,\ldots,k_{q-1})\succ(\ell_1,\ldots,\ell_{q-1})$
		means that at the first coordinate where the two vectors differ, the former vector has the larger entry.
			
		\begin{prop}\label{lem:triangular}
			Fix an integer \(q\ge1\) and a real number \(0<\rho<1/2\). There exist constants $c=c(q,\rho)>0$ and $C=C(q,\rho)<\infty$ with the following property.
			Let \(A(k_1,\ldots,k_q)\), \(1\le k_1,\ldots,k_q\le L\), be real coefficients satisfying
			\begin{align}
				&\distZ{A(k,k,\ldots,k)}\ge\rho \qquad(1\le k\le L) \label{eq:tri1}, \\
				\text{and} \quad &A(k_1,\ldots,k_q)=0
				\quad\text{if}\quad (k_1,\ldots,k_{q-1})\succ(k_q,\ldots,k_q).
				\label{eq:tri2}
			\end{align}
				Let \(U_{j,k}\) be as in Lemma~\ref{lem:GCS}. Then
			\begin{equation}\label{eq:triestimate}
				\left|\mathbb E e\!\left(\sum_{k_1,\ldots,k_q=1}^{L} A(k_1,\ldots,k_q) \prod_{j=1}^{q}U_{j,k_j}\right)\right|
				\leq C e^{-cL}.
			\end{equation}
		\end{prop}
		Recall that $\|\cdot \|$ denotes the distance to $\Z$. The proof is based on a more clever way to condition than a simple freezing procedure.
		\begin{proof}
			We first prove a one-step propagation estimate. Let \(U_1,\ldots,U_L\) be independent with the same law as \(U_{j,k}\) from above, and let \(\mathcal G\) be a sigma-algebra independent of them. Consider for $1 \leq k \leq L$ random variables $D_k$ which are \(\mathcal G\)-measurable and \(R_k\) which shall be  measurable with respect to the generated $\sigma$-algebra $ \sigma(\mathcal G, U_1,\ldots,U_{k-1})$.
			Let \(G\subset\{1,\ldots,L\}\) be \(\mathcal G\)-measurable (to be more precise, $G$ is a set-valued random variable), and suppose that 
			\[\distZ{D_k}\ge\delta \qquad(k\in G)\]
			for some \(0<\delta<1/2\). Then
			\begin{equation}\label{eq:propagation}
				\mathbb P\left(|\{k\in G:\distZ{R_k+U_kD_k}\ge\delta/3\}|
				<\frac{|G|}{8} \ \middle|\ \mathcal G \right)
				\leq e^{-c_0|G|},
			\end{equation}
			with an absolute constant \(c_0>0\). 
			Indeed,  consider  $k \in G$ and condition on $\sigma(\mathcal G, U_1,\ldots,U_{k-1})$. If the three points
			\[ R_k,\qquad R_k+D_k,\qquad R_k-D_k\]
			were within distance \(<\delta/3\) of \(\mathbb Z\), then \(D_k\) would be within distance \(<2\delta/3\) of \(\mathbb Z\), contradicting \(\distZ{D_k}\ge\delta\). Hence at least one of the three possible values has distance at least $\delta/3$ from $\Z$. In view of the law of \(U_k\), the conditional probability of the event $\distZ{R_k+U_kD_k} \geq \delta/3$ is at least $\frac14$.
			We abbreviate \(I_k :=\mathbf 1_{\{\|R_k+U_kD_k\|\ge \delta/3\}}\). Conditionally on
			\(\mathcal G\) the set \(G=\{k_1<\cdots<k_m\}\) is fixed, and we have thus established the conditional bound 
			\[ \mathbb P(I_{k_r}=1\mid \mathcal G,U_1,\ldots,U_{k_r-1})\geq 1/4
			\qquad (1\le r\le m).
			\]
			Thus the adapted sequence \((I_{k_r})_{r=1}^m\) stochastically dominates an
			i.i.d. Bernoulli sequence \((Z_r)_{r=1}^m\) with success probability $\frac14$, whence
			\[ \mathbb P\left(\sum_{k\in G} I_k<|G|/8\mid\mathcal G\right)
			\leq \mathbb P\left(\sum_{r=1}^{|G|}Z_r<|G|/8\right).\]
			A standard concentration estimate, e.g. using Hoeffding's inequality, yields the desired bound from  \eqref{eq:propagation}.
				
			We turn to the proof the lemma. The case \(q=1\) is immediate, since
			$$\mathbb E e(A(k)U_{1,k}) =\frac12+\frac14e(A(k))+\frac14e(-A(k))=\cos^2(\pi A(k)),$$
			where we used the trigonometric identity $1 + \cos(2x) = 2 \cos^2(x)$. Therefore
			$$ \prod_{k=1}^{L}\left|\mathbb E e(A(k)U_{1,k})\right|
				\leq \left(\sup_{\distZ{x}\ge\rho}\cos^2(\pi x)\right)^L
				\leq e^{-cL}.$$
			Assume \(q\ge2\). For \(0\le r\le q-1\), we define
			\begin{align}
				\label{eq:Hk}
				H_k^{(r)} :=\sum_{\substack{1\le i_{r+1},\ldots,i_{q-1}\le L\\
				(\,\underbrace{k,\ldots,k}_{r\ \mathrm{times}},i_{r+1},\ldots,i_{q-1}\,)\preceq (k,\ldots,k)}} A(\underbrace{k,\ldots,k}_{r\ \mathrm{times}},
				i_{r+1},\ldots,i_{q-1},k)  \,
				\times \prod_{\ell=r+1}^{q-1}U_{\ell,i_\ell}.
			\end{align}
			When \(r=q-1\), we interpret the empty product as $1$ and obtain $H_k^{(q-1)}=A(k,\ldots,k)$.
			The assumption \eqref{eq:tri1} reads 
			\begin{equation}\label{eq:tri1'}
				\distZ{H_k^{(q-1)}}\ge\rho \qquad(1\le k\le L).
			\end{equation} 
			For $0\le r\le q-2$, the lexicographic condition \eqref{eq:tri2} implies the recursive form
			\begin{equation}\label{eq:Hrecurs}
				H_k^{(r)}=R_k^{(r)}+U_{r+1,k}H_k^{(r+1)},
			\end{equation}
			where \(R_k^{(r)}\) is measurable with respect to the variables
				\[ \{U_{r+1,j}:j<k\} \quad\text{and}\quad
				\{U_{\ell,j}:r+2\le \ell\le q-1,\ 1\le j\le L\}.\]
			This can be seen as follows. After fixing the first $r$  equal to $k$, the next index is either \(<k\) or equal to \(k\) since the case \(>k\) vanishes by \eqref{eq:tri2}. The case \(=k\) gives the term $U_{r+1,k}H_k^{(r+1)}$; and the case \(<k\) is absorbed into \(R_k^{(r)}\). The claims on the measurability are immediate.
				
			We start from \eqref{eq:tri1'} and the idea is to apply the one-step estimate \eqref{eq:propagation} successively for $r=q-2,q-3,\ldots,0$ to obtain the final bound 
			\begin{equation}\label{eq:Hind}
				\mathbb P\left( \left|\{k:\distZ{H_k^{(0)}}\ge\delta\}\right|\geq\kappa L
				\right) \geq1-C_1e^{-c_1L}.
			\end{equation}
			We now make the successive propagation argument precise.  Define the sequences 
			\[ \delta_{q-1}:=\rho,\qquad \delta_r:=\delta_{r+1}/3 \quad (0\le r\le q-2),
			\]
			and
			\[ \kappa_{q-1}:=1,\qquad \kappa_r:=\kappa_{r+1}/8 \quad (0\le r\le q-2).
			\]
			For $0\le r\le q-1$, we further put the the active index set put
			\[ G_r:=\{1\le k\le L:\|H_k^{(r)}\|\ge \delta_r\}.
			\]
			We claim that, for every \(0\le r\le q-1\),
			\begin{equation}\label{eq:propaind}
			\mathbb P\bigl(|G_r|\ge \kappa_r L\bigr)
			\ge 1-C_r e^{-c_r L},
			\end{equation}
			with constants \(C_r,c_r>0\) depending only on \(q\) and \(\rho\).
			
			For  $r=q-1$, this follows from $\|H_k^{(q-1)}\|\geq\rho$ for all
			$k$.  Suppose now that $1\le r+1\le q-1$ and that the claim \eqref{eq:propaind} has already been proved at level $r+1$.  Let
			$\mathcal A_{r+1}:=\{|G_{r+1}|\ge \kappa_{r+1}L\}$ be the event of interest.
			We condition on the sigma-algebra generated by all variables
			\[ \{U_{\ell,j}: r+2\le \ell\le q-1,\ 1\le j\le L\}.
			\]
			Under this conditioning, \(G_{r+1}\) and the coefficients
			\(D_k:=H_k^{(r+1)}\) are fixed, while the variables
			\((U_{r+1,k})_{k=1}^L\) remain independent and have the required law.Moreover, by the recursive representation
			$H_k^{(r)}=R_k^{(r)}+U_{r+1,k}H_k^{(r+1)}$,
			the term \(R_k^{(r)}\) is measurable with respect to the conditioned sigma-algebra together with \(U_{r+1,1},\ldots,U_{r+1,k-1}\).
			Therefore the one-step estimate \eqref{eq:propagation}, applied with active set \(G_{r+1}\),
			gives on \(\mathcal A_{r+1}\)
			\[ \mathbb P\left(|G_r\cap G_{r+1}|<\frac{|G_{r+1}|}{8}
			\,\middle|\,\{U_{\ell,j}: r+2\le \ell\le q-1\} \right)
			\leq e^{-c_0|G_{r+1}|}
			\leq e^{-c_0\kappa_{r+1}L}.\]
			On the complementary event we have
			\[ |G_r|\geq |G_r\cap G_{r+1}|
			\geq \frac{|G_{r+1}|}{8} \geq \frac{\kappa_{r+1}}8 L
			= \kappa_r L .
			\]
			Consequently,
			\[ \mathbb P(|G_r|<\kappa_r L)
			\leq \mathbb P(\mathcal A_{r+1}^c)
			+ e^{-c_0\kappa_{r+1}L}
			\leq C_{r+1}e^{-c_{r+1}L} + e^{-c_0\kappa_{r+1}L}.
			\]
			This proves \eqref{eq:propaind} at level $r$ with some absolute constants $c_r$ and $C_r$.  Iterating down to $r=0$ gives the desired bound 
			\[ \mathbb P\bigl(|G_0|\ge\kappa_0L\bigr) \geq 1-C_1e^{-c_1L}.
			\]
			This completes the proof of \eqref{eq:Hind}.
			
			To conclude the proof of the proposition, we condition on the first $q-1$ groups of variables $U_{j,k}$. By assumption, i.e. by \eqref{eq:tri2}, the coefficient of $U_{q,k}$ is exactly $H_k^{(0)}$. Hence
			\[ \sum_{k_1,\ldots,k_q} A(k_1,\ldots,k_q)\prod_{j=1}^{q}U_{j,k_j}
				= \sum_{k=1}^{L}U_{q,k}H_k^{(0)}.\]
			Taking the $U_q$ average, we obtain the exact identity
			\[ \left| \mathbb E_{U_q} e\!\left(\sum_{k=1}^{L}U_{q,k}H_k^{(0)}\right)
			\right| =\prod_{k=1}^{L}\cos^2(\pi H_k^{(0)}).\]
			On the event from $\eqref{eq:Hind}$, at least $\kappa L$ of these factors are bounded by $\tau :=\sup_{\distZ{x}\ge\delta}\cos^2(\pi x)<1$
			Thus, the conditional expectation is at most $\tau^{\kappa L}$ on that good event and of course at most $1$ always. Averaging and using the probabilistic estimate from $\eqref{eq:Hind}$  gives
			\[\left| \mathbb E e\!\left(\sum_{k_1,\ldots,k_q}
				A(k_1,\ldots,k_q)\prod_{j=1}^{q}U_{j,k_j}\right)\right|
				\leq \tau^{\kappa L}+C_1e^{-c_1L}
				\leq C e^{-cL}. \]
			This proves the proposition.
		\end{proof}
			
		Unfortunately, it is impossible to restrict the proof to perfect triangular structures as required in Proposition~\ref{lem:triangular}. To overcome this issue, we combine the previous two results to obtain a perturbative estimate for almost triangular multilinear forms. 
			
		\begin{lemma}\label{lem:finite-bias}
			Fix \(q\ge2\) and \(0<\rho<1/2\). There exist constants \(c=c(q,\rho)>0\) and \(C=C(q,\rho)<\infty\) such that the following holds.	
			Let \(Q\) be a real polynomial of degree at most $q$ in independent fair Bernoulli variables \(Z_{j,k}\), \(1\le j\le q\), \(1\le k\le L\), written in multilinear form. Let \(A(k_1,\ldots,k_q)\) be the coefficient of $Z_{1,k_1}\cdots Z_{q,k_q}.$
			Assume the diagonal estimate
			\begin{equation}\label{eq:diag}
			\distZ{A(k,k,\ldots,k)}\geq \rho \qquad(1\le k\le L).
			\end{equation}
			We define the future mass
			\begin{equation}\label{eq:mass}
				m := \sum_{\substack{1\le k_1,\ldots,k_q\le L\\
					(k_1,\ldots,k_{q-1})\succ(k_q,\ldots,k_q)}}
				|A(k_1,\ldots,k_q)|.
			\end{equation}
			Then,
			\begin{equation}\label{eq:finite-bias}
				\left|\mathbb E e(Q(Z))\right|
				\le
				C\bigl(e^{-cL}+m \bigr)^{2^{-q}}.
			\end{equation}
		\end{lemma}
			
		\begin{proof}
			Let
			$$T_A(U) := \sum_{k_1,\ldots,k_q=1}^{L}
				A(k_1,\ldots,k_q) \prod_{j=1}^{q}U_{j,k_j}. $$
			Let \(A_0\) be the coefficients obtained from \(A\) by setting to zero all future coefficients occurring in \eqref{eq:mass}, and let \(T_{A_0}\) be the corresponding multilinear form. Since \(|U_{j,k}|\le1\), we have pointwise
			$|T_A(U)-T_{A_0}(U)|\leq m. $
			Therefore
			\begin{equation}\label{eq:A-A0}
				\left|\mathbb E e(T_A(U))-\mathbb E e(T_{A_0}(U))\right|
				\leq 2\pi m.
			\end{equation}
			The tensor \(A_0\) satisfies the hypotheses of Proposition \ref{lem:triangular} since its diagonal coefficients are those of $A$, and coefficients not respecting the lexicographic condition  vanish. Hence
			\[ \left|\mathbb E e(T_{A_0}(U))\right| \le C e^{-cL}\]
			for some absolute constants $C,c >0$ depending only on $q$ and $\rho$. Together with \eqref{eq:A-A0} and Lemma~\ref{lem:GCS}, this gives
			$$ \left|\mathbb E e(Q(Z))\right|^{2^q} \leq
				C e^{-cL}+2\pi m. $$
			Changing constants completes the proof.
			\end{proof}
			
			\medskip

	\subsection{A deterministic triangular block construction.}
			
		The next lemma constructs many \(d\)-tuples of digit positions with prescribed sum. Its role is to arrange the top-degree coefficients of $P$ in an almost triangular pattern. This will ultimately allow us to use our decay estimate from Proposition~\ref{lem:triangular}. To increase the readability, we employ two common asymptotic notions in the following. We write for two sequences $a_n \ll_{x} b_n$ if there exists a constant $C(x)$ - which may depend on the parameter (or family of parameters) $x$ - such that $a_n \leq C(x) b_n$ for all $n$. Similarly, we write $a_n \asymp_{x} b_n$, if there exists two constants $c(x), C(x)$ possibly depending on $x$ such that $c(x) b_n \leq a_n \leq C(x) b_n$ for all $n$.
			
		\begin{lemma}\label{lem:blocks}
			Fix integers \(d\ge2 , r\), and a positive real number $\beta>0$. 
			We define the functions
			$s :=\lfloor \beta n\rfloor+r$.
			and  $L:=\lfloor(\log n)^2\rfloor$.
			Given an integer $R$, we define the scale function
			$$H:= H(R) =\left\lfloor \frac{s}{L^R}\right\rfloor.$$
			Provided \(R=R(d)\) is chosen large enough, for
			all sufficiently large \(n\), there exist pairwise disjoint integer intervals
			\[ I_{j,k}\subset\mathbb N, \qquad 1\le j\le d,\quad 1\le k\le L,
			\]
			and an absolute constant $c_d$ (only depending on $d$)  with the following properties. \\	
			First, every \(i\in I_{j,k}\) satisfies
			\begin{equation}\label{eq:indasymp}
				i\asymp_{\beta,d} n.
			\end{equation}
			Second, for each \(k\),
			\begin{equation}\label{eq:sumasymp}
				|\left\{(i_1,\ldots,i_d)\in I_{1,k}\times\cdots\times I_{d,k}:
				i_1+\cdots+i_d=s\right\}|
				\asymp_d H^{d-1}.
			\end{equation}
			Third, suppose that for every \(1\le k\le L\) one chooses a tuple
			$(i_{1,k},\ldots,i_{d,k}) \in I_{1,k}\times\cdots\times I_{d,k}$
			such that
			\begin{equation}\label{eq:tuplesum}
			i_{1,k}+\cdots+i_{d,k}=s.
			\end{equation}
			Then, whenever $(k_1,\ldots,k_{d-1})\succ(k_d,\ldots,k_d)$,
			one has
			\begin{equation}\label{eq:offdiag}
			i_{1,k_1}+\cdots+i_{d,k_d}\ge s+c_d H.
			\end{equation}
		\end{lemma}
			
		\begin{proof}
			We choose fixed positive numbers \(\theta_1,\ldots,\theta_d\) with sum $1$ and pairwise distinct values. For instance, one may take
			\[ \theta_j=\frac{2j}{d(d+1)} \qquad(1\le j\le d).\]
			The numbers $\theta_j$ give rise to corresponding digit positions
			\[ u_j:=\lfloor \theta_js\rfloor \quad(1\le j<d),
			\qquad u_d=s-u_1-\cdots-u_{d-1}.\]
			Note that \(u_j=\theta_js+O_d(1)\) for all $1 \leq j \leq d$ and that the $u_j$'s are separated from one another by a fixed positive proportion of $s$. The idea is that $u_j$ is the leading order of the center of $I_{j,k}$. Of course to guarantee disjointness, one needs to perturb $u_j$ by some $k$-dependent subleading term. These are constructed in the next step.
				
			We choose integer weights \(W_1,\ldots,W_{d-1}\) recursively as follows. Let
			\[ W_{d-1} :=100dH,\]
			and, for \(1\le j<d-1\), we set
			\begin{equation}\label{eq:Wweight}
				W_j := 100dL\bigl(W_{j+1}+\cdots+W_{d-1}+H\bigr).
			\end{equation}
			Then 
			\[ W_j\ll_d H L^{d-1-j}, \qquad LW_1\ll_d H L^{d-1}.\]
			We fix the exponent $R$ with \(R>d+2\) such that 
			\begin{equation}\label{eq:Rcond}
				LW_1=o(s).
			\end{equation}
				
			We are finally ready to define the integer centers of the intervals $I_{j,k}$. For $1\le k\le L$, let
			\[ a_{j,k} :=u_j+kW_j \quad(1\le j<d), \]
			and $a_{d,k}$ is defined to ensure the summability condition,
			\[ a_{d,k} :=s-\sum_{j=1}^{d-1} a_{j,k}.\]
			By construction
			\begin{equation}\label{eq:sumcond}
				a_{1,k}+\cdots+a_{d,k}=s.
			\end{equation}
			Because of \eqref{eq:Rcond}, all centers remain within fixed positive proportions of \(s\). Since the original macroscopic centers \(u_1,\ldots,u_d\) are separated by a fixed positive proportion of $s$,   the intervals
			\[I_{j,k} :=[a_{j,k}-H,a_{j,k}+H]\cap\mathbb Z \]
			are pairwise disjoint for different $j \neq j'$ and all sufficiently large $n$. For $1 \leq j \leq d-1$, we have by construction $W_j > 100 H$, which is good enough to ensure that $ I_{j,k} \cap I_{j,k'} = \emptyset$ for $k \neq k'$. The final case $j = d$, follows by inspection of the summability constraint since $|a_{d,k} - a_{d,k'}| > |k-k'|W_1$. This proves the disjointness and the above argument already reveals that every element of every interval is bounded uniformly from below by $\alpha n$ with some $\alpha > 0$. The upper bound is trivial and, hence, \eqref{eq:indasymp} follows.  
				
			The second claim \eqref{eq:sumasymp} follows from \eqref{eq:sumcond}. Indeed, we write \(i_j=a_{j,k}+x_j\), with \(|x_j|\le H\). The condition \(i_1+\cdots+i_d=s\) becomes
			$x_1+\cdots+x_d=0$. This becomes a simple counting problem
			Since $x_d$ is uniquely determined by the previous elements, there are at most \(O_d(H^{d-1})\) choices in total. Conversely, all choices 
			\[ |x_1|,\ldots,|x_{d-1}|\le \frac{H}{2d} \]
			are admissible because \(|x_d|\le H/2\). Thus \eqref{eq:sumasymp} holds.
				
			It remains to verify the crucial triangular property. We consider tuples $(i_{1,k},\ldots,i_{d,k}) \in I_{1,k}\times\cdots\times I_{d,k}$ satisfying \eqref{eq:tuplesum}. Similarly to the previous step, we write $i_{j,k} =a_{j,k}+x_{j,k}$ 
			with $|x_{j,k}|\leq H$.
			We fix a $d$-element $(k_1,\ldots,k_d)$ with
			$(k_1,\ldots,k_{d-1})\succ(k_d,\ldots,k_d).$
			Let \(j_0\) be the first index \(1\le j_0\le d-1\) for which \(k_{j_0}>k_d\). The summability condition yields
			$i_{d,k_d}=s-\sum_{j=1}^{d-1}i_{j,k_d}.$
			Therefore
			\begin{align}\label{eq:telescop}
				i_{1,k_1}+\cdots+i_{d,k_d}-s &=
				\sum_{j=1}^{d-1} \bigl(i_{j,k_j}-i_{j,k_d}\bigr).	
			\end{align}
			For \(j<j_0\), one has \(k_j=k_d\), so the corresponding summands in \eqref{eq:telescop} vanish. For \(j=j_0\),
			\[ i_{j_0,k_{j_0}}-i_{j_0,k_d} \geq W_{j_0}-2H.\]
			For \(j>j_0\),
			\[ \left|i_{j,k_j}-i_{j,k_d}\right| \leq L W_j+2H. \]
			By the recursive construction \eqref{eq:Wweight}, the positive contribution \(W_{j_0}\) dominates all possible later negative contributions. More precisely,
			\[ W_{j_0}-2H
				- \sum_{j=j_0+1}^{d-1}(LW_j+2H) \geq c_d^{j_0} H\]
			with some absolute constant $c_d^{j_0} > 0$ for large enough $n$.
			Thus \eqref{eq:offdiag} holds with $c_d := \min_{1 \leq j_0 < d} c_d^{j_0}$.
			\end{proof}
			
			\medskip
			
		\subsection{ Extraction of sparse fair Bernoulli variables}
			
		In this little section, we explain how we can transfer results on fair Bernoulli such as Lemma~\ref{lem:triangular} to our highly biased setting of interest. We abbreviate $\alpha :=\frac{d-1}{d}$.
		Let $\Gamma > 0$ be the logarithmic exponent from Theorem~\ref{thm:main} (which we still need to fix in the upcoming proof). We define the sequence
		\begin{equation}\label{eq:lambda}
		\lambda_j := \begin{cases} \frac12(\log j)^\Gamma j^{-\alpha} & \text{ if } j > j_0, \\
			0 & \text{ else,}
		\end{cases}
		\end{equation}
		with a cutoff $j_0$ which is fixed in a moment. 
		Since $\min(p_j,1-p_j)\ge p_j(1-p_j)=v_j$, the 
		condition $v_j\ge (\log j)^{\Gamma} j^{-\alpha}$  implies sufficiently large \(j\),
		\[\lambda_j\le \min(p_j,1-p_j).\]
		The cutoff $j_0$ is chosen such that the above estimate holds for all $j > j_0$. This elementary observation allows us to realize the biased Bernoullis  \(\xi_j\) as follows. Let \(\chi_j\) be Bernoulli with parameter \(\lambda_j\). If \(\chi_j=1\), let \(\eta_j\) be a fair Bernoulli variable. If \(\chi_j=0\), let \(\zeta_j\) be Bernoulli with parameter
		\[ q_j=\frac{p_j-\lambda_j/2}{1-\lambda_j}.\]
		The inequality \(\lambda_j\le\min(p_j,1-p_j)\) guarantees \(q_j\in[0,1]\). We take all auxiliary variables independent, and define
		\begin{equation}\label{eq:decomp}
			\xi_j=\chi_j\eta_j+(1-\chi_j)\zeta_j.
		\end{equation}
		Then
		\[\mathbb P(\xi_j=1) = \lambda_j\cdot\frac12+(1-\lambda_j)q_j =
		p_j.\]
		Thus this construction has exactly the original distribution. The variables with \(\chi_j=1\) lead to the sparse fair bits $\eta_j$. Our bounds will be applied to exactly these random collection of indices endowed with fair Bernoulli's. We use this notation throughout the next section.

		\subsection{The Fourier decay estimate}
			
		After all these preparations, we are finally ready to prove the key Fourier estimate. This is arguably the technically most challenging proof of this work. It requires a combination of rather distinct considerations. 
			
		\begin{prop}\label{prop:correlation}
			Let \(M>0\). If the exponent $\Gamma$ in \eqref{eq:sharp} is sufficiently large in terms of degree $d$ and $M$, then for every  base $b\ge2$ and every nonzero integer \(h\in\mathbb Z\setminus\{0\}\),
			\begin{equation}\label{eq:fourier}
				\left| \mathbb E e\!\left(h(b^n-b^m)P(X)\right)
				\right| \ll_{P,b,h,M} (\log n)^{-M}
			\end{equation}
			uniformly for $0\le m<n$.
			\end{prop}
		It is important to note that this estimate does only depend on the larger exponent $n$. 
		\begin{proof}
			We fix throughout the proof the base $b\geq2$ and the probing integer $h\neq 0$. Let us first introduce all needed parameters. We abbreviate $D :=d!a_d\neq 0$ and set the $b$-adic scale $\beta :=\log_2 b$.
			We choose an integer $r=r(P,h)$ such that
			\begin{equation}\label{eq:r}
			\frac1{16}\le |hD|2^{-r}<\frac18.
			\end{equation}
			As in Lemma~\ref{lem:triangular}, we set $s:=\lfloor\beta n\rfloor+r$.	
			It is also convenient to write as in previous proofs $t := h(b^n-b^m)$, but one should keep in mind that $t$ depends on $n$ and $m$. 
			As a first little step, we want to rewrite the bounds \eqref{eq:r} in terms of $s$ and $t$.  We have
			\[ tD2^{-s} = hD(b^n-b^m)2^{-\lfloor\beta n\rfloor-r}       
			 = hD2^{-r} 2^{\beta n-\lfloor\beta n\rfloor}(1-b^{m-n}).\]
			Now
			\[1\le 2^{\beta n-\lfloor\beta n\rfloor}<2,
				\qquad \frac12\le 1-b^{m-n}<1.
			\]
			Therefore \eqref{eq:r} implies
			\begin{equation}\label{eq:nonresonant}
				\frac1{32}\le |tD2^{-s}|<\frac14.
			\end{equation}
			We can rewrite this as 
			\begin{equation}\label{eq:rho}
				\distZ{tD2^{-s}}\ge \rho, \qquad \rho :=\frac1{32},
			\end{equation}
			where the notation suggest that this estimate will correspond to the assumption from Proposition~\ref{lem:triangular} and Lemma~\ref{lem:finite-bias}.
				
			We want to apply Lemma \ref{lem:blocks} with the parameters $s, \beta, r$ from above. We fix $R = R(d)$ as in Lemma \ref{lem:blocks} and we accordingly set $L=\lfloor(\log n)^2\rfloor$ and
			\begin{equation}\label{eq:H} 
			H=\left\lfloor\frac{s}{L^R}\right\rfloor
			\asymp_{b} \frac{n}{(\log n)^{2R}}.
			\end{equation}
			We obtain pairwise disjoint intervals \(I_{j,k}\) with the crucial separation property \eqref{eq:offdiag}
			For \(1\le k\le L\), let \(\mathcal T_k\) be the level-sets from Lemma~\ref{lem:blocks}, i.e.
			\[ \mathcal T_k := \left\{
				(i_1,\ldots,i_d)\in I_{1,k}\times\cdots\times I_{d,k}:
				i_1+\cdots+i_d=s \right\}. \]
			Then Lemma \ref{lem:blocks} also yields the asymptotics
			\begin{equation}\label{eq:Tkasymp}
				|\mathcal T_k|\asymp_d H^{d-1}.
			\end{equation}
				
			Let \(Y_k\) count the selected tuples in \(\mathcal T_k\):
			\begin{equation}\label{eq:Yk}
				Y_k := \sum_{(i_1,\ldots,i_d)\in\mathcal T_k} \chi_{i_1}\cdots\chi_{i_d}.
			\end{equation}
			We stress that we work with the selector Bernoulli variables $\chi_j$ not with the actual Bernoullis $\xi_j$. All indices appearing in the intervals \(I_{j,k}\) are \(\asymp_{b,d}n\). Hence, recalling the cut decay sequence $\lambda_j$ from \eqref{eq:lambda},
			\begin{equation}\label{eq:lamasymp}
				\lambda_i\asymp_{b,d,\Gamma}(\log n)^\Gamma n^{-\alpha}
				\qquad(i\in I_{j,k}).
			\end{equation}
			This allows us to estimate the expectation of $Y_k$. Indeed using \eqref{eq:Tkasymp}, we get
			\begin{align}
				\mu_k:=\mathbb E Y_k
					\asymp_{b,d,\Gamma} H^{d-1}\bigl((\log n)^\Gamma n^{-\alpha}\bigr)^d        
					\asymp_{b,d,\Gamma}
					\left(\frac{n}{(\log n)^{2R}}\right)^{d-1}
					(\log n)^{\Gamma d}n^{-(d-1)}                           
					= (\log n)^{\Gamma d-2R(d-1)}.
			\end{align}
			This calculation reveals how to choose the exponent $\Gamma$: we fix $\Gamma$ large enough such that	
			\begin{equation}\label{eq:Gamma}
			\kappa := \Gamma d-2R(d-1) >M+4.
			\end{equation}
			In particular, we then have uniformly $\mu_k \to \infty$ as $n \to \infty$. We have now all parameters at hand and this finishes the preparatory part of the proof and we turn to the main estimates.
				
			In this direction, want to control  the probability $ \mathbb P(Y_k=0)$. We argue via a second-moment analysis. Let us write
			$I_\tau :=\prod_{j=1}^{d}\chi_{i_j}$ for $\tau=(i_1,\ldots,i_d)\in\mathcal T_k$
			Then $Y_k=\sum_{\tau\in\mathcal T_k}I_\tau$ and 
			\[ \operatorname{Var}(Y_k)
			= \sum_{\tau,\tau'\in T_k} \operatorname{Cov}(I_\tau,I_{\tau'}).
			\]
			We separate the diagonal terms, the disjoint pairs, and the overlapping pairs.
			For the diagonal contribution we simply use
			\[ \sum_{\tau\in T_k}\operatorname{Var}(I_\tau)
			\leq \sum_{\tau\in T_k}\mathbb E I_\tau = \mu_k.
			\]
			If two distinct tuples \(\tau,\tau'\in T_k\) have no coordinate in common, then \(I_\tau\) and \(I_{\tau'}\) are independent, and hence their covariance vanishes.  If they have exactly \(u\) common coordinates, \(1\le u\le d-1\), then the number of ordered pairs \((\tau,\tau')\) of this kind is $O_d(H^{2d-2-u})$.
			Therefore by \eqref{eq:lamasymp},
			\[ \mathbb E(I_\tau I_{\tau'}) \ll_{b,d,\Gamma}
				\bigl((\log n)^\Gamma n^{-\alpha}\bigr)^{2d-u}.
			\]
			Since $ \operatorname{Cov}(I_\tau,I_{\tau'})
			\leq \mathbb E(I_\tau I_{\tau'})$,
			we obtain in total
			\begin{equation}\label{eq:varY}
				\operatorname{Var}(Y_k) \leq
				\mu_k + C_{b,d,\Gamma} H^{2d-2}((\log n)^\Gamma n^{-\alpha}\bigr)^{2d} \sum_{u=1}^{d-1}
				(H(\log n)^\Gamma n^{-\alpha}\bigr)^{-u}.
			\end{equation}
			Using again \(\mu_k\asymp H^{d-1}((\log n)^\Gamma n^{-\alpha}\bigr)^{d}\), this becomes
			\begin{equation}\label{eq:varY2}
			\operatorname{Var}(Y_k)
			\leq \mu_k + C_{b,d,\Gamma}\mu_k^2
			\sum_{u=1}^{d-1}(H(\log n)^\Gamma n^{-\alpha})^{-u}.
			\end{equation}
			For all sufficiently large $n$, we have $\sum_{u=1}^{d-1}(H(\log n)^\Gamma n^{-\alpha})^{-u} \ll_{b,d,\Gamma}
			n^{-1/(2d)}$. Collecting the previous estimates, we obtain
			\begin{equation}\label{eq:varY3}
				\frac{\operatorname{Var}(Y_k)}{\mu_k^2} \ll_{b,d,\Gamma}
				(\log n)^{-\kappa}
				+n^{-1/(2d)}.
			\end{equation}
			Chebyshev's inequality gives
			\begin{equation}
			\mathbb P(Y_k=0) \leq  \mathbb P(|Y_k-\mu_k|\ge \mu_k)
			\leq \frac{\operatorname{Var}(Y_k)}{\mu_k^2}
			\ll_{b,d,\Gamma} (\log n)^{-\kappa} + n^{-1/(2d)}
			\end{equation}
			for all sufficiently large \(n\). Hence, by the union bound,
			\begin{equation}\label{eq:union}
				\mathbb P(Y_k>0\text{ for every }1\le k\le L)
				\ge
				1-
				O_{b,d,\Gamma}\!\left((\log n)^{2-\kappa}+L n^{-1/(2d)}\right).
			\end{equation}
			Because of \eqref{eq:Gamma}, the exceptional probability in \eqref{eq:union} is
			\begin{equation}\label{eq:good}
				O_{b,d,\Gamma,M}\bigl((\log n)^{-M}\bigr).
			\end{equation}
				
			We work in the following on the good event \(\Omega_{\mathrm{good}}\) on which \(Y_k>0\) for every $1\leq k\leq L$. On \(\Omega_{\mathrm{good}}\), choose, for each \(k\), one active tuple
			\begin{equation}\label{eq:chosen}
				(i_{1,k},\ldots,i_{d,k})\in\mathcal T_k
			\end{equation} 
			by a deterministic rule depending only on the selector variables \((\chi_j)\). For instance, one may choose the lexicographically first active tuple. By construction
			\begin{equation}\label{eq:tupleadd}
				i_{1,k}+\cdots+i_{d,k}=s
			\end{equation}
			and
			\[ \chi_{i_{1,k}}=\cdots=\chi_{i_{d,k}}=1.\]
				
			The idea is that on $\Omega_{\mathrm{good}}$ we have many active sites and thus the fair Bernoulli variables $\eta_j$ come into play. As we have already seen in Proposition~\ref{lem:triangular}, if fair Bernoullis follow a triangular structure, we can obtain the desired cancellations.  Let $\mathcal{A}$ be the $\sigma$-algebra generated by the following data:	
				\[
				\begin{gathered}
					\text{all selector variables }(\chi_j),\\
					\text{all background variables }(\zeta_j),\\
					\text{all fair variables }(\eta_j)
					\text{ except those equal to }
					\eta_{i_{j,k}}.
				\end{gathered}
				\]
			In the following we condition on $\mathcal{A}$. Under this conditioning, the variables
			\[ Z_{j,k}:=\eta_{i_{j,k}}, \qquad 1\le j\le d,\quad 1\le k\le L,
			\]
			remain independent fair Bernoulli variables.
				
			For the conditioned value of $X$, we may write
			\begin{equation}\label{eq:Xcond}
				X=X_0+\sum_{j=1}^{d}\sum_{k=1}^{L}2^{-i_{j,k}}Z_{j,k},
			\end{equation}
			where $X_0$ is $\mathcal A$-measurable. Here it is important that we have chosen the tuples in \eqref{eq:chosen} in a measurable way. We can treat $X_0$ as constant in the following. The expression $tP(X)$ 
			becomes a real polynomial $Q$ of degree at most $d$ in the $dL$ Bernoulli variables $Z_{j,k}$. We pass to its unique multilinear representation and want to compute its coefficient of
			$Z_{1,k_1}Z_{2,k_2}\cdots Z_{d,k_d}$
			Recall that the intervals \(I_{j,k}\) from Lemma~\ref{lem:blocks} are pairwise disjoint. Thus, all selected digit positions \(i_{j,k}\) are distinct. Therefore only the leading term $a_dX^d$ of $P$ can contribute to this coefficient. A direct computation shows
			\begin{equation}\label{eq:multicoff} 
				A(k_1,\ldots,k_d) = t\,d!a_d\,2^{-(i_{1,k_1}+\cdots+i_{d,k_d})}
				= tD\,2^{-(i_{1,k_1}+\cdots+i_{d,k_d})}.
			\end{equation}
			This justifies a posteriori the definition of $D$. 
		    The diagonal coefficients are readily determined. Indeed \eqref{eq:tupleadd} gives
			\[ A(k,k,\ldots,k) =tD2^{-s}.\]
			Now the use of \eqref{eq:rho} becomes clear as it implies the diagonal estimate,
			\begin{equation}\label{eq:diagest}
				\distZ{A(k,k,\ldots,k)}\ge\rho=\frac1{32} \qquad(1\le k\le L).
			\end{equation}
			Of course, $A$ is not triangular and we cannot apply Proposition~\ref{lem:triangular} directly. Instead, we need to refer to Lemma~\ref{lem:finite-bias} and this requires a bound on the future mass. 	
			 If $(k_1,\ldots,k_{d-1})\succ(k_d,\ldots,k_d)$,
			then Lemma \ref{lem:blocks}, applied to the chosen tuples \eqref{eq:chosen}, gives
			\[	i_{1,k_1}+\cdots+i_{d,k_d}\ge s+c_d H.\]
			Combining this with \eqref{eq:multicoff} and \eqref{eq:nonresonant},
			\[ |A(k_1,\ldots,k_d)| \leq
				|tD|2^{-s}2^{-c_d H}\leq \frac14\,2^{-c_d H}.\]
			Therefore the future mass $m$ from Lemma~\ref{lem:finite-bias}  satisfies
			\begin{equation}\label{eq:fut}
				m \leq L^d2^{-c_dH}
				\ll (\log n)^{-M2^d}
			\end{equation}
			for all sufficiently large \(n\). The first bound follows from the fact that there at most $L^d$ choices for the $k$-tuples. And the second estimate is trivial since $H$ grows polynomially up to a logarithmic suppression. On  \(\Omega_{\mathrm{good}}\), we have
			\begin{equation}
				\left| \mathbb E\left(e(tP(X))
					\,\middle|\,\mathcal A \right) \right|
					\leq C_d\bigl(e^{-c_dL}+m\bigr)^{2^{-d}}
					\ll_{P,b,h,M} (\log n)^{-M}.
			\end{equation}
			We used that $L$ grows logarithmically and hence $e^{-c_dL}$ decays faster than the logarithm.	
			Finally,
			\[ \left|\mathbb E e(tP(X))\right|
					\leq \mathbb P(\Omega_{\mathrm{good}}^c)
					+ \mathbb E\left[
					\mathbf 1_{\Omega_{\mathrm{good}}}
				\left|\mathbb E\left(e(tP(X))\,\middle|\,\mathcal A\right) \right| \right]                                              
			\ll_{P,b,h,M} (\log n)^{-M},
				\]
			 This completes the proof.
			\end{proof}
			
			\medskip
			
			\subsection{Completing the proof}
			
			The rest of the proof of Theorem~\ref{thm:main} is now a rather routine application of Weyl's criterion.
			
			\begin{proof}[Proof of Theorem~\ref{thm:main}]
				Choose a  number $M>1$, for example \(M=2\). Choose $\Gamma_0$ large enough such that  Proposition\ref{prop:correlation} applies with this value of $M$. Note that the value $\Gamma_0$ depends only on the degree of the polynomial $P$. 
				We fix an integer base \(b\ge2\), a non-zero integer $h$ and  define the random variables
				$$ Y_n=e\!\left(hb^nP(X)\right),
				\qquad S_N=\sum_{n=0}^{N-1}Y_n. $$
				Of course \(|Y_n|=1\), and
				\[
				\begin{aligned}
					\mathbb E|S_N|^2
					&=
					\sum_{m,n=0}^{N-1}
					\mathbb E e\!\left(h(b^n-b^m)P(X)\right)            \\
					&=
					N
					+
					2\operatorname{Re}
					\sum_{0\le m<n<N}
					\mathbb E e\!\left(h(b^n-b^m)P(X)\right).
				\end{aligned}
				\]
				The second term is estimated via Proposition \ref{prop:correlation}. We obtain 
				$$ \mathbb E|S_N|^2 \ll_{P,b,h}
				N+ \sum_{n=2}^{N-1}\frac{n}{(\log n)^M}. $$
				To estimate the partial sums, one may proceed by a dyadic decomposition or an integral comparison combined with repeated partial integration to obtain
				$$ \sum_{n=2}^{N-1}\frac{n}{(\log n)^M}
				\ll \frac{N^2}{(\log N)^M}. $$
				Since \(M>1\), we get 
				\[ \sum_{N=2}^{\infty} \frac1{N^3}\mathbb E|S_N|^2
					\ll_{P,b,h} \sum_{N=2}^{\infty}\frac1{N^2}
					+ \sum_{N=3}^{\infty}\frac1{N(\log N)^M} <\infty.
				\]
				By Lemma \ref{lem:DEL},
				\[ \frac1N\sum_{n=0}^{N-1} e\!\left(hb^nP(X)\right)
				\longrightarrow0 \qquad\text{almost surely.}
				\]
				This holds for fixed $b,h$. Taking the countable intersection over all $b \geq 2$ and \(h\in\mathbb Z\setminus\{0\}\), Weyl's criterion gives that $P(X)$ is normal to all integer bases $b$ and, thus, absolutely normal.
			\end{proof}

		\section{Proof of Theorem~\ref{thm:necessary}}\label{sec:proof3}
		
		We fix throughout this section a constant $C > 0$ and assume always $p_n\sim C n^{-(d-1)/d}$.  Before starting with the actual proof, let us sketch the main idea to derive Theorem~\ref{thm:necessary}. In this sparse regime, it seems natural that powers $X^k$ for $1 \leq k < d$ do not have enough carry mixing. Hence, the normality of an integer-valued polynomial $Q(X)$ should only depend on the behavior of the monomial $X^d$. Typically, under our decay assumption there is only a finite interaction. This is sparse enough to expect a Poissonian limit and the correct way to think about $X^d$ is that
		$$ X^d \simeq \sum_{n = 1}^\infty Z_n 2^{-n}$$
		where the random variables $Z_n$ are independent Poisson with some parameter $\lambda > 0$ at least for large enough $n$. This is not a binary expansion since $Z_k$ may take values larger than $1$ and it is not completely obvious if such a number can still be normal. However, if we consider the model case that the $Z_k$ are perfectly i.i.d. Poisson random variables, the Fourier transform $e(2^N \sum_{n = 1}^\infty Z_n 2^{-n})$ is readily computed and shown to be non-zero. By Weyl's criterion one obtains that $X^d$ is at least not almost surely normal.

		Making this heuristics precise, requires some effort. Our first lemma records the asymptotics of $k$-fold interactions and will be used for many upcoming estimates.
		
		\begin{lemma}\label{lem:asympdfold}
			For $1\leq k\leq d$, we define
			\begin{equation}
			U_k(N) :=\sum_{\substack{i_1,\ldots,i_k\ge 1\\ i_1+\cdots+i_k=N}}
			p_{i_1}\cdots p_{i_k}.
			\end{equation}
			Then
			\begin{equation}
			U_k(N) \sim C^k \frac{\Gamma(1/d)^k}{\Gamma(k/d)} N^{k/d-1},
			\end{equation}
			where $\Gamma(x)$ is the usual Gamma-function. In particular, $U_d(N)\to C^d\Gamma(1/d)^d$  while
			$U_k(N)\to0$  for every \(k<d\). \\
			Moreover, if \(1\le r<d\) and \(m_1,\ldots,m_r\) are fixed positive
			integers, then
			\begin{equation}
			V_{\mathbf m}(N)
			:=
			\sum_{\substack{i_1,\ldots,i_r\ge 1\\
					m_1i_1+\cdots+m_ri_r=N}}
			p_{i_1}\cdots p_{i_r}
			=
			O\bigl(N^{r/d-1}\bigr)
			=
			o(1).
			\end{equation}
			All estimates are uniform under bounded shifts of \(N\).
		\end{lemma}
		
		\begin{proof}
			Let us abbreviate $ \alpha := (d-1)/d<1$ such that  \(p_n\sim Cn^{-\alpha}\). A simple 
			Riemann-sum estimate gives
			\[
			\begin{aligned}
				U_k(N)
				&=
				C^k
				\sum_{\substack{i_1,\ldots,i_k\ge1\\ i_1+\cdots+i_k=N}}
				i_1^{-\alpha}\cdots i_k^{-\alpha}
				(1+o(1))                                                  \\
				&\sim
				C^k N^{k-1-k\alpha}
				\int_{\substack{x_j>0\\ x_1+\cdots+x_k=1}}
				x_1^{-\alpha}\cdots x_k^{-\alpha}\,d\sigma(x),
			\end{aligned}
			\]
			and note that $k-1-k\alpha = \frac{k}{d} - 1$. The integral is finite because \(\alpha<1\) and gives rise to the Beta function which may be written as 
			\[
			\int_{\substack{x_j>0\\ x_1+\cdots+x_k=1}}
			x_1^{-\alpha}\cdots x_k^{-\alpha}\,d\sigma(x)
			= \frac{\Gamma(1-\alpha)^k}{\Gamma(k(1-\alpha))}
			= \frac{\Gamma(1/d)^k}{\Gamma(k/d)}. \]
			This completes the proof for the asymptotics of $U_k(N)$ and the assertion on the limits $N \to \infty$ are immediate.
			
			A similar argument applies to the affine hyperplane
			$m_1x_1+\cdots+m_rx_r=1$.
			Indeed,
			\[
			\begin{aligned}
				V_{\mathbf m}(N)
				\sim C^r
				\sum_{\substack{i_1,\ldots,i_r\ge1\\
						m_1i_1+\cdots+m_ri_r=N}}
				i_1^{-\alpha}\cdots i_r^{-\alpha} (1+o(1))                        
				&\sim A(r)
				N^{r-1-r\alpha}
				= A(r)
				N^{r/d-1}
			\end{aligned}
			\]
			for some explicit constant $A(r)$. Since \(r<d\), this tends to \(0\).  The claim on the bounded shifts is a direct consequence of the proof.
		\end{proof}
		
		For \(N\ge1\), let \(\mathcal E_N\) be the set of \(d\)-element subsets
		\(E\subset\mathbb N\) satisfying $\sum_{i\in E} i=N.$
		For such a set put the $d$-fold product
		\[
		w(E)=\prod_{i\in E}p_i
		\]
		and we further define the random diagonal count
		\[
		A_N=\sum_{E\in\mathcal E_N}\prod_{i\in E}\xi_i .
		\]
		The key probabilistic insight is that $A_N$ has a Poissonian limit under the decay assumption for $p_n$. 
		
		\begin{lemma}\label{lem:Poisson}
			For every fixed \(K\ge1\), the following convergence in distribution holds:
			\begin{equation}
			(A_{N+1},\ldots,A_{N+K})
			\xrightarrow{D}
			(Z_1,\ldots,Z_K),
			\end{equation}
			where \(Z_1,\ldots,Z_K\) are independent Poisson random variables of mean $\lambda := \frac{C^d}{d!}\Gamma(1/d)^d$.
		\end{lemma}
		
		The proof is quite straightforward as we essentially expand all factorial moments explicitly. The main idea is that under our scaling assumptions, overlaps between $A_{N+l}$ and $A_{N+k}$ are negligible resulting in the independence. The same mechanism also leads to the Poisson convergence for each vector component. Unfortunately,  this leads to quite lengthy combinatorial estimates.
		\begin{proof}
			Let us first check that the first moments coincide. By independence,
			$\mathbb \E A_N =
			\sum_{E\in\mathcal E_N}w(E)$.
			To relate this sum to $U_d(N)$ from Lemma~\ref{lem:asympdfold}, we note that the definition of $U_d(N)$ takes the ordering into account and includes repeated indices.  The ordered distinct \(d\)-tuples contribute \(d!\mathbb E A_N\). Moreover, tuples with repeated indices are
			bounded by finitely many sums of the form \(V_{\mathbf m}(N)\) with fewer
			than $d$ distinct variables.  By Lemma~\ref{lem:asympdfold}, the repeated part only contributes 
			\(o(1)\).  Hence
			\[
			d!\,\mathbb E A_N
			=
			U_d(N)+o(1)
			\to
			C^d\Gamma(1/d)^d,
			\]
			and therefore \(\mathbb E A_N\to\lambda\).
			
			To control higher joint moments, we need to show that overlaps are negligible.  We fix two distinct integers $a \neq b \in \Z$ and
			claim that
			\begin{equation} \Delta_N(a,b) :=
			\sum_{\substack{E\in\mathcal E_{N+a},\,
					F\in\mathcal E_{N+b}\\
					E\cap F\ne\emptyset}}
			\prod_{i\in E\cup F}p_i
			\longrightarrow 0 .
			\end{equation}
			If $a$ or $b$ is negative, we only consider large enough $N$ such that $\Delta_N$ is well-defined. Note that $|E \cap F| = d$ is impossible since $a \neq b$ 
			It suffices to consider pairs with \(|E\cap F|=s\), where
			\(1\le s\le d-1\).  Let \(S=E\cap F\), \(|S|=s\), and write
			\(\sigma=\sum_{i\in S}i\).  If we ignore that $F \setminus S \cap E \setminus S = \emptyset$,  we only increase the number of possible $d$-elements.
			Thus the contribution for this value of \(s\) is bounded by
			\[ 	\sum_{\substack{E\in\mathcal E_{N+a},\,
					F\in\mathcal E_{N+b}\\
					|E\cap F|\ne\emptyset}}
			\prod_{i\in E\cup F}p_i
			\leq \sum_{\sigma} U_s(\sigma)\, U_{d-s}(N+a-\sigma)\, U_{d-s}(N+b-\sigma),
			\]
			where we set $U_k(x) \equiv 0$ if $x < 0$.  By Lemma~\ref{lem:asympdfold}
			\[ U_s(\sigma) = O(\sigma^{s/d-1}), \qquad U_{d-s}(M) = O(M^{-s/d}).
			\]
			The contribution of small values of $\sigma$ and $N + a - \sigma$ is negligible and hence there is some absolute constant $D$ such that for $N$ large enough 
			\[ \sum_{\sigma} U_s(\sigma)\, U_{d-s}(N+a-\sigma)\, U_{d-s}(N+b-\sigma) \leq D  \sum_{\sigma = 1}^{N} \sigma^{s/d-1} (1+N-\sigma)^{-s/d} (1+N-\sigma)^{-s/d}.
			\]
			We split the sum into \(\sigma\le N/2\) and \(\sigma>N/2\). The first contribution is easily estimated by
			\[	\sum_{\sigma\le N/2} \sigma^{s/d-1}N^{-2s/d} = O(N^{-s/d}),
			\]
			whereas
			\[
			\sum_{\sigma>N/2}
			N^{s/d-1}(1+N-\sigma)^{-2s/d}
			=
			\begin{cases}
				O(N^{-s/d}), & 2s<d,\\
				O(N^{-1/2}\log N), & 2s=d,\\
				O(N^{s/d-1}), & 2s>d.
			\end{cases}
			\]
			All three bounds tend to \(0\).  Hence \(\Delta_N(a,b)\to0\). Note that for $a = b$, we still have by the same argument
			\[ \Delta_N^{\circ}(a,a) :=
			\sum_{\substack{E\in\mathcal E_{N+a},\,
					F\in\mathcal E_{Nab}\\
					E\cap F\ne\emptyset}, \, E \neq F}
			\prod_{i\in E\cup F}p_i
			\longrightarrow 0
			\]
			if we exclude the diagonal term.
			
		 We move to our specific case of interest. It is convenient to 
	    abbreviate for 
		\(a\in\{1,\ldots,K\}\)
		\[
		\mathcal E_a(N):=\mathcal E_{N+a},
		\qquad
		A_a(N):=A_{N+a}.
		\]
		Thus
		\[
		A_a(N)=\sum_{E\in\mathcal E_a(N)} X_E,
		\qquad
		X_E:=\prod_{i\in E}\xi_i.
		\]
		Then $w(E):=\mathbb E X_E=\prod_{i\in E}p_i$ and we have already shown that $S_a(N):=\sum_{E\in\mathcal E_a(N)}w(E) \to \lambda$
	    for each fixed \(a\). Let us record two further elementary
		facts.
		First,
		\begin{equation}
		\max_{E\in\mathcal E_a(N)} w(E)\to0.
		\label{eq:max-weight-zero}
		\end{equation}
		is immediate
		
		Second, we need a uniform extension bound.   A direct consequence of Lemma~\ref{lem:asympdfold} is the estimate
		diagonal estimates proved above imply
		\[
		\sup_{M\in\mathbb Z} U_q(M)<\infty
		\qquad (0\le q\le d).
		\]
		Consequently, a simple decomposition implies that for every fixed integer \(M_0\ge0\), there is a constant
		\(B_{M_0}<\infty\) such that, for every set \(U\subset\mathbb N\) with
		\(|U|\le M_0\),
		\begin{equation}
		\sum_{E\in\mathcal E_a(N)} \prod_{i\in E\setminus U}p_i \leq B_{M_0},
		\label{eq:extension-bound-clean}
		\end{equation}
		uniformly in \(N\) and \(a\in\{1,\ldots,K\}\).  
		
	     We now finish the proof of the joint Poisson convergence by computing joint
		factorial moments.  We write $(x)_r=x(x-1)\cdots(x-r+1)$ for any real number $x$.  Fix non-negative integers
		\(r_1,\ldots,r_K\), and we define $R=r_1+\cdots+r_K$.
		Let
		\[
		\mathcal P
		:=
		\{(a,j):1\le a\le K,\ 1\le j\le r_a\}.
		\]
		For \(\rho=(a,j)\in\mathcal P\), we write \(a(\rho)=a\).
		An application of the inclusion-exclusion principle shows
		\[
		(A_a(N))_{r_a}=
		\sum_{\substack{E_{a,1},\ldots,E_{a,r_a}\in\mathcal E_a(N)\\
				E_{a,1},\ldots,E_{a,r_a}\ \text{distinct}}}
		X_{E_{a,1}}\cdots X_{E_{a,r_a}}.
		\]
		Hence, we may write the product of interests as
		\[
		\prod_{a=1}^K (A_a(N))_{r_a}
		=
		\sum_{\mathbf E\in\Omega_N}
		\prod_{\rho\in\mathcal P} X_{E_\rho},
		\]
		where \(\Omega_N\) is the set of all choices
		\[
		\mathbf E=(E_\rho)_{\rho\in\mathcal P},
		\qquad
		E_\rho\in\mathcal E_{a(\rho)}(N),
		\]
		such that for each fixed \(a\), the edges $E_{(a,1)},\ldots,E_{(a,r_a)}$
		are distinct.  Therefore
		\begin{equation}
		\mathbb E\prod_{a=1}^K (A_a(N))_{r_a}v= \sum_{\mathbf E\in\Omega_N}
		\prod_{i\in\bigcup_{\rho\in\mathcal P}E_\rho}p_i.
		\label{eq:factorial-expansion-clean}
		\end{equation}
		By the above overlap estimate, one expects that only disjoint configuration contribute in the limit $N \to \infty$. 
		We call such a configuration \(\mathbf E\in\Omega_N\) good, that is, $E$ is \textit{good} if all its selected
		edges \(E_\rho\), \(\rho\in\mathcal P\), are pairwise disjoint.  Otherwise
		we call it \textit{bad}.  
		
		We first show that the contribution of bad configurations
		to \eqref{eq:factorial-expansion-clean} is \(o(1)\).
		Fix two distinct positions \(\rho,\tau\in\mathcal P\).  We estimate the
		contribution of configurations with
		$E_\rho\cap E_\tau\ne\emptyset$.
		Choose \(E_\rho\) and \(E_\tau\) first.  Their contribution to the union weight is
		$\prod_{i\in E_\rho\cup E_\tau}p_i$.
		Now add the remaining \(R-2\) edges one at a time.  If the current union of
		previously chosen edges is \(U\), then adding a new edge \(E\) multiplies
		the union weight by
		\[
		\prod_{i\in E\setminus U}p_i.
		\]
		Since the total number of chosen vertices is always at most \(dR\), the
		extension bound \eqref{eq:extension-bound-clean}, with \(M_0=dR\), shows
		that the total contribution of each additional edge is bounded by a
		constant depending only on \(d,R,K\).  Hence
		\[
		\begin{aligned}
			\sum_{\substack{\mathbf E\in\Omega_N\\
					E_\rho\cap E_\tau\ne\emptyset}}
			\prod_{i\in\bigcup_{\sigma\in\mathcal P}E_\sigma}p_i
			&\leq C_{d,R,K}
			\Delta_N(a(\rho),a(\tau))
			=
			o(1).
		\end{aligned}
		\]
		if $a(\rho) \neq a(\tau)$ and otherwise the bound has to be replaced by $\Delta_N^{\circ}$.
		There are only finitely many pairs \(\rho,\tau\).  Therefore the full
		contribution of bad configurations is \(o(1)\).  Thus
		\begin{equation}
		\mathbb E\prod_{a=1}^K (A_a(N))_{r_a}
		= \sum_{\substack{\mathbf E\in\Omega_N\\ \mathbf E\ \mathrm{good}}}
		\prod_{\rho\in\mathcal P}w(E_\rho) +o(1).
		\label{eq:good-configurations-clean}
		\end{equation}
		
		It remains to evaluate the good sum.  Consider the unrestricted product $\prod_{a=1}^K S_a(N)^{r_a}$.
		Expanding it gives
		\[
		\prod_{a=1}^K S_a(N)^{r_a} =
		\sum_{\mathbf E\in\Omega_N^{\mathrm{all}}} \prod_{\rho\in\mathcal P}w(E_\rho),
		\label{eq:unrestricted-product-clean}
		\]
		where \(\Omega_N^{\mathrm{all}}\) is the set of all choices
		\(E_\rho\in\mathcal E_{a(\rho)}(N)\), with no distinctness or disjointness
		condition. A similar argument as for \eqref{eq:good-configurations-clean} shows that bad configurations and non-distinct configurations are again negligible due to the overlap estimate. We arrive at
		\begin{equation}
		\sum_{\substack{\mathbf E\in\Omega_N\\ \mathbf E\ \mathrm{good}}}
		\prod_{\rho\in\mathcal P}w(E_\rho)
		= \prod_{a=1}^K S_a(N)^{r_a}
		+o(1).
		\label{eq:good-sum-product-clean}
		\end{equation}
		
		Combining \eqref{eq:good-configurations-clean} and
		\eqref{eq:good-sum-product-clean}, we obtain
		\[
		\mathbb E\prod_{a=1}^K (A_a(N))_{r_a}
		=
		\prod_{a=1}^K S_a(N)^{r_a}
		+o(1).
		\]
		Since \(S_a(N)\to\lambda\), it follows that
		\[
		\mathbb E\prod_{a=1}^K (A_{N+a})_{r_a}
		\longrightarrow
		\prod_{a=1}^K \lambda^{r_a}.
		\]
		These are exactly the joint factorial moments of \(K\) independent
		Poisson random variables with common mean \(\lambda\). The Poisson moments satisfy the Carleman condition and thus uniquely determine its distribution. Hence, the claimed convergence in distribution follows.
		\end{proof}
		
		In our next result, we exploit the Poisson convergence to deduce a non-trivial limit of the Fourier transform of $Q(X)$ along the dyadic sequence $2^N$.
		\begin{prop}\label{prop:fourier-obstruction}
			Let \(Q\in\mathbb Z[x]\) be an integer-valued polynomial of  degree $d$, say
			\[ Q(x)=a_dx^d+a_{d-1}x^{d-1}+\cdots+a_0, \qquad a_d\ne0.
			\]
			Then
			\begin{equation}
			\lim_{N\to\infty}
			\mathbb E\, e(2^NQ(X))=L_Q,
			\end{equation}
			where
			\[
			L_Q
			=
			\prod_{\ell=1}^{\infty}
			\exp\left(
			\lambda\bigl(e(a_dd!2^{-\ell})-1\bigr)
			\right).
			\]
			In particular, $0<|L_Q|<1$.
		\end{prop}
		
		\begin{proof}
			We start from the simple identity for the monomial
			$$ 	X^k=\sum_{M\geq k} B_M^{(k)}2^{-M}$$
			with 
			\[ B_M^{(k)} =
			\sum_{\substack{i_1,\ldots,i_k\ge1\\
					i_1+\cdots+i_k=M}} \xi_{i_1}\cdots\xi_{i_k}
			\]
			for $1 \leq k \leq d$. Note that is just a binomial expansion and in particular $B_M^{(k)}$ does not need to coincide with the binary digits. 
			We still have modulo \(1\),
			\begin{equation}
			2^NQ(X)\equiv \sum_{\ell\ge1}2^{-\ell} \sum_{k=1}^{d} a_k B_{N+\ell}^{(k)}
			\pmod 1,
			\end{equation}
			because the terms with \(M\le N\) give integers.
			
			If \(k<d\), then \(\mathbb E B_M^{(k)}\to0\). This follows from the observation that $B_M^{(k)}$ can be written as finite linear combination in terms of 
			\(V_{\mathbf m}(M)\) with fewer than \(d\) distinct variables. 
			By Lemma~\ref{lem:asympdfold} these terms converge to $0$ and  are uniformly bounded. \\
			For \(k=d\), one may revisit the proof of Lemma~\ref{lem:Poisson} to see that  
			\[ B_M^{(d)}=d!A_M+R_M, \qquad \mathbb E|R_M|\to0.\]
			Therefore, for each fixed integer $K$, Lemma~\ref{lem:Poisson} further implies 
			\[ \left(B_{N+1}^{(d)},\ldots,B_{N+K}^{(d)}\right)
			\xrightarrow{D} \left(d!Z_1,\ldots,d!Z_K\right),
			\]
			where \(Z_1,\ldots,Z_K\) are independent Poisson variables with parameter $\lambda := \frac{C^d}{d!}\Gamma(1/d)^d$.
			
			Since the whole vector $\left(B_{N+1}^{(d)},\ldots,B_{N+K}^{(d)}\right)$ converges in distribution, the Cramer-Wold device implies that this also holds for all linear combinations of the vector components.   It follows that the finite truncations satisfy
			\begin{equation}
			\sum_{\ell=1}^{K}2^{-\ell}
			\sum_{k=1}^{d} a_k B_{N+\ell}^{(k)}
			\xrightarrow{D} a_dd!\sum_{\ell=1}^{K}2^{-\ell}Z_\ell .
		    \end{equation}
			Since the expectations of
			\(B_M^{(k)}\) are uniformly bounded for $1\le k\le d$,
			we obtain a uniform tail estimate of the form
			\[ \sup_N \mathbb E \sum_{\ell>K}2^{-\ell} \sum_{k=1}^{d}|a_k|B_{N+\ell}^{(k)}
			\leq D  \sum_{\ell>K}2^{-\ell} \to 0, \]
			where $D$ is an absolute constant depending on $Q$ and the decay rate $C$. 
			In total, we obtain the convergence in distribution for the dyadic series 
			\begin{equation}
			2^NQ(X) \pmod1 \xrightarrow{D}
			a_dd!\sum_{\ell\ge1}2^{-\ell}Z_\ell \pmod1.
			\end{equation}
			Since \(e(\cdot)\) is a continuous function on $\mathbb R/\mathbb Z$, the definition of weak convergence directly yields
			\begin{align*}
				\lim_{N\to\infty}\mathbb E\,e(2^NQ(X))
				&=
				\mathbb E\,
				e\left(a_dd!\sum_{\ell\ge1}2^{-\ell}Z_\ell\right)  
				=
				\prod_{\ell=1}^{\infty}
				\mathbb E\,e(a_dd!2^{-\ell}Z_\ell),
			\end{align*}
			where we used the independence in the last step
			For a Poisson\((\lambda)\) variable $ \mathbb E e(tZ)
			= \exp\bigl(\lambda(e(t)-1)\bigr)$.
			Hence
			\[ L_Q = \prod_{\ell=1}^{\infty}
			\exp\left(
			\lambda\bigl(e(a_dd!2^{-\ell})-1\bigr)
			\right). \]
			The product defining $L_Q$ converges absolutely because
			$ \sum_{\ell\ge1}|e(a_dd!2^{-\ell})-1|<\infty$.
			In particular, \(L_Q\ne0\).  For large enough $\ell$, we have
			\(e(a_dd!2^{-\ell})\ne1\) and this yields $|L_Q|<1$.
		\end{proof}
		
		The Fourier obstruction in form of Proposition~\ref{prop:fourier-obstruction} is now the key to show the non-normality of $Q(X)$. Indeed, in view of Weyl's criterion $Q(X)$ being normal almost surely would imply by dominated convergence that $L_Q = 0$. This simple argument does not exclude, however, the possibility that $Q(X)$ is normal with probability less than one. To show the almost sure non-normality, we employ a martingale argument.
		
		\begin{proof}[Proof of Theorem~\ref{thm:necessary}]
			We use the standard fact that multiplication by a non-zero rational number preserves normality to a fixed integer base. Hence, we may restrict ourselves to $Q \in \Z[X]$ of degree $d$.  
			Let \(a_d\) be the leading coefficient of \(Q\), and let \(L_Q\ne0\) be the non-zero constant from Proposition~\ref{prop:fourier-obstruction}.
			
			To make use of the martingale convergence theorem later, we record a conditional version of Proposition~\ref{prop:fourier-obstruction}. Consider the for $m \in \N$ the generated $\sigma$-algebra
			$ \mathcal F_m=\sigma(\xi_1,\ldots,\xi_m)$.
			Conditioned on \(\mathcal F_m\), the random number $X$ takes the form
			\[
			X=2^{-m}(K+Z_m),
			\]
			where \(K\in\{0,\ldots,2^m-1\}\) is a fixed integer and
			$Z_m=\sum_{n\ge1}\xi_{m+n}2^{-n}$ is independent from $\mathcal{F}_m$ 
			with tail variables \((\xi_{m+n})_{n\ge1}\) are independent and our decay assumptions reads
			$\mathbb P(\xi_{m+n}=1)=p_{m+n}\sim Cn^{-\alpha}$.
			The conditioning gives naturally rise to a polynomial acting on $Z$ as follows
			\begin{equation}
			Q_{m,K}(z)= 2^{md}Q\left(\frac{K+z}{2^m}\right).
			\end{equation}
			Then \(Q_{m,K}\in\mathbb Z[z]\).  Indeed, if
			\(Q(x)=\sum_{r=0}^d a_rx^r\), then
			\[
			2^{md}Q\left(\frac{K+z}{2^m}\right)
			=
			\sum_{r=0}^d a_r2^{m(d-r)}(K+z)^r,
			\]
			which has integer coefficients.  Moreover, the leading coefficient is again \(a_d\). This allows us to apply 
		    Proposition~\ref{prop:fourier-obstruction} conditionally to the tail variable \(Z_m\) and the polynomial
			\(Q_{m,K}\). We get, almost surely,
			\begin{equation*}
			\lim_{N\to\infty}
			\mathbb E\left[ e\left(2^NQ_{m,K}(Z_m)\right) \,\middle|\,\mathcal F_m \right] = L_Q .
			\end{equation*}
			In particular, by construction we have 
			\[ \lim_{N\to\infty} \mathbb E\left[ e\left(2^{md}2^N Q(X)\right)
			\,\middle|\,\mathcal F_m
			\right] = L_Q .\]
			By Cesaro averaging, we also have 
			\begin{equation}\label{eq:cesaro} \lim_{N\to\infty} \mathbb E\left[
			\frac1N\sum_{n=1}^N e\left(2^{md}2^nQ(X)\right)
			\,\middle|\,\mathcal F_m
			\right] = L_Q .
			\end{equation}
			
			Now, we conclude the proof. Let \(\mathcal N_2\) be the event that \(Q(X)\) is normal in base $2$.
			Under the event \(\mathcal N_2\), Weyl's criterion gives, for every non-zero integer \(h\), the convergence
			$\frac1N\sum_{n=1}^N e(h2^nQ(X))\to0.$
			In particular,
			\[ A_N^{(m)} := \frac1N\sum_{n=1}^N e(2^{md}2^nQ(X))
			\to0 \qquad\text{on }\mathcal N_2.\]
			By conditional dominated convergence,
			\[ \mathbb E\left[ A_N^{(m)}\mathbf 1_{\mathcal N_2}
			\,\middle|\,\mathcal F_m \right] \to0. \]
			By \eqref{eq:cesaro}, we obtain on the complementary event
			\[ \mathbb E\left[A_N^{(m)}\mathbf 1_{\mathcal N_2^c}\,\middle|\,\mathcal F_m \right] \to L_Q. \]
			Therefore, almost surely $ |L_Q| \leq \mathbb P(\mathcal N_2^c\mid\mathcal F_m)$ or, equivalently,
			\[ \mathbb P(\mathcal N_2\mid\mathcal F_m)
			\leq 1-|L_Q| \qquad\text{a.s.}\]
			and this holds for every natural number $m$.  Since
			\(\mathcal N_2\) is measurable with respect to
			\(\sigma(\xi_1,\xi_2,\ldots)\), Levy's martingale convergence theorem gives
			\[ \mathbb P(\mathcal N_2\mid\mathcal F_m)
			\to \mathbf 1_{\mathcal N_2} \qquad\text{a.s.}
			\]
			The last two displays are almost sure statements on the random variable $\mathbf 1_{\mathcal N_2}$ and contradict each other (since $|L_Q|>0$) unless
			\[ \mathbf 1_{\mathcal N_2}=0 \qquad\text{a.s.}
			\]
			Thus, \(Q(X)\) is almost surely not normal in base $2$.
		\end{proof}
		
		\paragraph{Acknowledgments.} The author thanks Christoph Aistleitner for an initial idea, which ultimately lead to this work. This work was funded by the Deutsche Forschungsgemeinschaft (DFG, German Research Foundation) - 558731723.

		\end{document}